\documentclass[a4paper,11pt]{article}

\usepackage[T1]{fontenc}
\usepackage[utf8]{inputenc}
\usepackage[english]{babel}					
\usepackage{lmodern}
\usepackage{amsmath,amsfonts,amssymb, mathrsfs}
\usepackage{amsthm}
\usepackage{mathtools}
\usepackage[dvipsnames]{xcolor}				%
\usepackage{mathscinet}						%

\usepackage{tikz}
\usepackage{pgfplots}
\usepackage{subfig}
\pgfplotsset{compat=1.12}					%

\usepackage{graphicx}
\usepackage[shortlabels]{enumitem}
\usepackage{braket}
\usepackage{booktabs}
\usepackage[font={small}, labelfont={sf,bf}, format={hang}]{caption}
\usepackage{subfig}
\usepackage[titletoc]{appendix}
\usepackage{hyperref}						%
\usepackage[autostyle]{csquotes}  			%
\usepackage[style=alphabetic,%
	isbn=false,%
    giveninits=true, backend=biber]{biblatex}

\usepackage[capitalise, noabbrev]{cleveref}
		\crefname{appsec}{Appendix}{Appendices}
\usepackage{refcount}
\usepackage{parskip}
\usepackage{tabularx}

\newtheorem{theorem}{Theorem}[section]
\newtheorem{lemma}[theorem]{Lemma}
\newtheorem{corollary}[theorem]{Corollary}
\newtheorem{proposition}[theorem]{Proposition}
\newtheorem*{theorem*}{Theorem}
\newtheorem*{claim*}{Claim}
\Crefname{subtheorem}{Subtheorem}{Subtheorem}
	\newtheorem{subtheorem}{Subtheorem}
\theoremstyle{definition}
\newtheorem{definition}[theorem]{Definition}%
\Crefname{impremark}{Remark}{Remarks}
	
\newtheorem*{example}{Example}

\theoremstyle{remark}
\newtheorem*{remark}{Remark}

\DeclarePairedDelimiter{\abs}{\lvert}{\rvert}

\DeclarePairedDelimiter{\pths}{(}{)}
\DeclarePairedDelimiter{\bkts}{[}{]}
\DeclarePairedDelimiter{\brcs}{\lbrace}{\rbrace}
\DeclarePairedDelimiter{\ango}{\langle}{\rangle}
\DeclarePairedDelimiter{\floor}{\lfloor}{\rfloor}
\newcommand{\inquotes}[1]{``#1''}					%

\newcommand{\numberset}{\mathbb}
\newcommand{\N}{\numberset{N}}
\newcommand{\Q}{\numberset{Q}}
\newcommand{\R}{\numberset{R}}
\newcommand{\Z}{\numberset{Z}}

\DeclareMathOperator{\id}{id}
\DeclareMathOperator{\Id}{Id}

\DeclareMathOperator{\diverg}{div}
\DeclareMathOperator{\vol}{Vol}
\DeclareMathOperator{\dist}{dist}

\DeclareMathOperator{\supp}{supp}

\DeclareMathOperator{\D}{d}
\DeclareMathOperator{\spann}{span}
\DeclareMathOperator{\card}{card}

\newcommand{\haus}{\mathscr{H}}  					%
\newcommand{\sphere}[1]{\mathbb{S}^{#1}}
\newcommand{\ball}[2]{B_{#2}{\pths{#1}}}
\newcommand{\clball}[2]{\bar{B}_{#2}{\pths{#1}}}
\newcommand{\varn}{\mathcal{N}}						%
\newcommand{\varm}{\mathcal{M}}						%
\newcommand{\sing}{\mathcal{S}}
\newcommand{\singu}{\mathcal{S}(u)}
\newcommand{\strat}[4][u]{\mathcal{S}^{#2}_{#3,#4}(#1)}

\newcommand{\drv}[1]{\frac{d}{d #1}}				%

\newcommand{\pdrvarg}[2]{\frac{\partial #2}{\partial #1}}

\newcommand{\sobd}[1]{W^{1,2}{\pths{#1,\varn}}}

\newcommand{\C}{\mathcal{C}}

\newcommand{\wonep}{W^{1,p}}
\makeatletter
\def\wonepN{\@ifstar\@wonepN\@@wonepN}
	\def\@wonepN#1{\wonep\pths*{#1,\varn}}
	\def\@@wonepN#1{\wonep\pths{#1,\varn}}
\makeatother
\newcommand{\cm}{x_{cm}}

\newcommand{\ene}{\mathcal{E}}			%
\newcommand{\fat}[2]{\mathcal{B}_{#2}{\pths{#1}}}
\newcommand{\half}[1][1]{\frac{#1}{2}}

\newcommand{\quint}{\frac{1}{5}}

\newcommand{\aff}[1]{\mathbf{H}^{#1}(\R^m)}
\newcommand{\grass}[1]{\mathbf{G}^{#1}(\R^m)}
\newcommand{\stset}{\; \middle| \;}
\newcommand{\hj}{{\hat{\jmath}}}

\newcommand{\laplace}{\Delta}

\renewcommand{\rho}{\varrho}
\renewcommand{\epsilon}{\varepsilon}
\renewcommand{\emptyset}{\varnothing}

\def\MySet#1{\expandafter\def\csname m#1\endcsname{\mathcal{#1}}%
			 \expandafter\def\csname mt#1\endcsname{\tilde{\mathcal{#1}}}%
			 \expandafter\def\csname m#1h\endcsname##1{\mathcal{#1}^{\pths*{##1}}}}
\def\ALLvec#1{\ifx#1\ALLvec\else\MySet#1\expandafter\ALLvec\fi}
\ALLvec ABCDEFGHIJKLMNOPQRSTUVWXYZ \ALLvec

\newcommand{\wolog}{without loss of generality}

\newcommand{\wrt}{with respect to}
\newcommand{\ie}{\emph{i.e.}}
\newcommand{\eg}{\emph{e.g.}}

\newcommand{\st}{such that}

\newcommand{\lhs}{left hand side}

\renewcommand{\iff}{if and only if}

\makeatletter		
		\let\save@mathaccent\mathaccent
		\newcommand*\if@single[3]{%
		  \setbox0\hbox{${\mathaccent"0362{#1}}^H$}%
		  \setbox2\hbox{${\mathaccent"0362{\kern0pt#1}}^H$}%
		  \ifdim\ht0=\ht2 #3\else #2\fi
		  }
		\newcommand*\rel@kern[1]{\kern#1\dimexpr\macc@kerna}
		\newcommand*\widebar[1]{\@ifnextchar^{{\wide@bar{#1}{0}}}{\wide@bar{#1}{1}}}
		\newcommand*\wide@bar[2]{\if@single{#1}{\wide@bar@{#1}{#2}{1}}{\wide@bar@{#1}{#2}{2}}}
		\newcommand*\wide@bar@[3]{%
		  \begingroup
		  \def\mathaccent##1##2{%
		    \let\mathaccent\save@mathaccent
		    \if#32 \let\macc@nucleus\first@char \fi
		    \setbox\z@\hbox{$\macc@style{\macc@nucleus}_{}$}%
		    \setbox\tw@\hbox{$\macc@style{\macc@nucleus}{}_{}$}%
		    \dimen@\wd\tw@
		    \advance\dimen@-\wd\z@
		    \divide\dimen@ 3
		    \@tempdima\wd\tw@
		    \advance\@tempdima-\scriptspace
		    \divide\@tempdima 10
		    \advance\dimen@-\@tempdima
		    \ifdim\dimen@>\z@ \dimen@0pt\fi
		    \rel@kern{0.6}\kern-\dimen@
		    \if#31
		      \overline{\rel@kern{-0.6}\kern\dimen@\macc@nucleus\rel@kern{0.4}\kern\dimen@}%
		      \advance\dimen@0.4\dimexpr\macc@kerna
		      \let\final@kern#2%
		      \ifdim\dimen@<\z@ \let\final@kern1\fi
		      \if\final@kern1 \kern-\dimen@\fi
		    \else
		      \overline{\rel@kern{-0.6}\kern\dimen@#1}%
		    \fi
		  }%
		  \macc@depth\@ne
		  \let\math@bgroup\@empty \let\math@egroup\macc@set@skewchar
		  \mathsurround\z@ \frozen@everymath{\mathgroup\macc@group\relax}%
		  \macc@set@skewchar\relax
		  \let\mathaccentV\macc@nested@a
		  \if#31
		    \macc@nested@a\relax111{#1}%
		  \else
		    \def\gobble@till@marker##1\endmarker{}%
		    \futurelet\first@char\gobble@till@marker#1\endmarker
		    \ifcat\noexpand\first@char A\else
		      \def\first@char{}%
		    \fi
		    \macc@nested@a\relax111{\first@char}%
		  \fi
		  \endgroup
		}
\makeatother
\newcommand{\ob}[2][0.48pt]{\setbox0=\hbox{\ensuremath{#2}}%
  \dimen0=\ht0\advance\dimen0 by0.2ex%
  \dimen3=0.17\ht0%
  \dimen1=\wd0\advance\dimen1 by-\dimen3\advance\dimen1 by-0.02em%
  \dimen2=\dimen1\advance\dimen2 by-0.03em%
  \box0%
  \kern-\dimen1\hbox{\rule[\dimen0]{\dimen2}{#1}}%
}

\newcommand{\mres}{%
  \text{\raisebox{.2ex}{\scalebox{1.25}{$\llcorner$}}}%
}

\newcommand\restr[2]{{%
  \left.\kern-\nulldelimiterspace %
  #1 %
  \vphantom{\big|} %
  \right|_{#2} %
  }}

\def\Xint#1{\mathchoice
{\XXint\displaystyle\textstyle{#1}}%
{\XXint\textstyle\scriptstyle{#1}}%
{\XXint\scriptstyle\scriptscriptstyle{#1}}%
{\XXint\scriptscriptstyle\scriptscriptstyle{#1}}%
\!\int}
\def\XXint#1#2#3{{\setbox0=\hbox{$#1{#2#3}{\int}$ }
\vcenter{\hbox{$#2#3$ }}\kern-.6\wd0}}

\def\dashint{\Xint-}

\newenvironment{stepizec}[1]{\begin{enumerate}[label=\textit{#1 \arabic*.} , ref={#1 \arabic*},align=left, 
leftmargin=0pt,
labelwidth =-\parindent,
topsep=0pt,
]}{\end{enumerate}}

\def\ConstantCounters#1{\newcounter{CustCount#1}}
\def\CountLetters#1{\ifx#1\CountLetters\else\ConstantCounters#1\expandafter\CountLetters\fi}
\CountLetters ABCDEFGHIJKLMNOPQRSTUVWXYZabcdefghijklmnopqrstuvxyz\CountLetters
\ConstantCounters{gamma}
\ConstantCounters{delta}
\ConstantCounters{xi}
\ConstantCounters{rho}

\newcommand{\hiddennconst}[2]{%
		\refstepcounter{CustCount#1}%
		\label{#2}%
		}

\newcommand{\nconst}[2]{%
		\refstepcounter{CustCount#1}%
		\ensuremath{{#1}_{\arabic{CustCount#1}}}%
		\label{#2}}

\newcommand{\oconst}[2]{%
		\ensuremath{{#1}_{\ref*{#2}}}}

\numberwithin{equation}{section}

\usepackage[bottom=3.5cm,top=3cm,left=3cm,right=3cm]{geometry}

\title{Quantitative regularity for $p$-minimizing maps through a Reifenberg Theorem}
\author{Mattia Vedovato}
\date{\today}

\begin{document}
\maketitle

\begin{abstract}
In this article we extend to generic $p$-energy minimizing maps between Riemannian manifolds a regularity result which is known to hold in the case $p=2$. We first show that the set of singular points of such a map can be \emph{quantitatively stratified}: we classify singular points based on the number of \emph{almost-symmetries} of the map around them, as done in \cite{Cheeger2013}. Then, adapting the work of Naber and Valtorta \cite{Naber2017}, we apply a Reifenberg-type Theorem to each quantitative stratum; through this, we achieve an upper bound on the Minkowski content of the singular set, and we prove it is $k$-rectifiable for a $k$ which only depends on $p$ and the dimension of the domain.
\end{abstract}

\tableofcontents
\section{Introduction}
The aim of this article is to study the singular set $\singu$ of maps that take values in a smooth Riemannian manifold and minimize the $p$-energy functional:
	\begin{equation}
	\ene_p(u)\doteq\int\abs*{\nabla u}^p
	\end{equation}
for $p\in\pths{1,\infty}$; we assume that $u:B\subset\varm\to\varn$, where $B$ is a geodesic ball in the Riemannian smooth manifold $\varm$, and $\varn$ is a smooth closed Riemannian manifold. Here by \emph{singular set} we mean the set of points where the map is not continuous. The information we obtain is twofold:
\begin{enumerate}
	\item First we obtain a Minkowski-type estimate on the singular set: if $u:B\subset\varm\to\varn$, and the total $p$-energy is bounded by $\Lambda$, then
		\begin{equation}
		\tag{ME}
		\vol\pths*{\singu\cap B'}\leq C(\varm,\varn,\Lambda,p)r^{\floor{p}+1},
		\end{equation}
		where $B'$ is in general a smaller geodesic ball. In particular, the Minkowski dimension of $\singu$ is at most $\dim\varm-\floor{p}-1$, and the upper Minkowski content (in $B'$) is bounded by a universal constant $C(\varm,\varn,\Lambda,p)$.
	\item Secondly, a rectifiability result is proved: we show that $\singu$ is $(\dim\varm-\floor{p}-1)$-rectifiable.
\end{enumerate}
In what follows, we set $m\doteq \dim\varm$. Our results complete the following scheme:
\begin{itemize}
\item $p=2$: in 1982, Schoen and Uhlenbeck proved through a dimension reduction argument that any $2$-energy minimizing map is $\C^{0,\alpha}$ outside of a set of Hausdorff dimension at most $m-3$ (see \cite{Schoen1982}). Furthermore, by standard elliptic regularity, the regularity outside the singular set can be improved to $\C^\infty$.
\item $p\in\pths{1,\infty}$: in 1987,  Schoen and Uhlenbeck's result was extended to $p$-harmonic maps by Hardt and Lin \cite{Hardt1987}. The best regularity one can achieve outside the singular set is $\C^{1,\alpha}$; so any $p$-energy minimizing map is $\C^{1,\alpha}$ outside of a set of Hausdorff dimension at most $m-\floor{p}-1$. Notice that the case $m\leq p$ was already completely solved here: in this case, there are no singular points, and the map is everywhere $\C^{1,\alpha}$. The only case worth studying is $m>p$.
\item $p=2$: in 2013, Cheeger and Naber \cite{Cheeger2013} proved that the singular set of a $2$-minimizing map with energy bounded by $\Lambda$ satisfies the following estimate:
	\begin{equation}\label{eq: cheeger naber}
		\vol\pths*{\singu\cap\ball{0}{1}}\leq C(\varm,\varn,\Lambda,\epsilon)r^{3-\epsilon}
	\end{equation}
	for any $\epsilon>0$; it is assumed that the dimension of the domain is at least $3$. This implies that the Minkowski dimension of $\singu$ is at most $m-3$, but gives no bound on the Minkowski content. Here a notion of \inquotes{quantitative stratification} of the singular set was introduced, and the result was obtained through a relatively simple covering of each singular stratum (and by making explicit the link between singular set and stratification).
\item $p\in\pths{1,\infty}$: in 2014, Naber, Valtorta and Veronelli \cite{Naber2014} extended the estimate \eqref{eq: cheeger naber} to $p$-minimizing maps: they showed that in this case 
	\begin{equation}
		\vol\pths*{\singu\cap\ball{0}{1}}\leq C(\varm,\varn,\Lambda,p,\epsilon)r^{\floor{p}+1-\epsilon}
	\end{equation}
	for any $\epsilon>0$. It is assumed that the dimension of the domain is greater than $p$: we have already noticed, however, that this is the only interesting case.
\item $p=2$: in 2017, Naber and Valtorta \cite{Naber2017} improved the estimate \eqref{eq: cheeger naber} for $2$-minimizing maps, removing the dependence on the parameter $\epsilon$: assuming that $m>2$, then the singular set of a $2$-minimizing map with energy bounded by $\Lambda$ satisfies
	\begin{equation}
	\vol\pths*{\singu\cap\ball{0}{1}}\leq C(\varm,\varn,\Lambda)r^{3};
	\end{equation}
	and thus the upper Minkowski content of $\singu$ is bounded by a constant $C$. Moreover, in the same article they showed that $\singu$ is actually $(m-3)$-rectifiable. The main idea to prove both the Minkowski estimate and rectifiability was to replace the simple covering argument of \cite{Cheeger2013} with a more refined one, which makes use of a suitable version of Reifenberg Theorem.
\end{itemize}
In this article, we adopt the same technique used in \cite{Naber2017}: after giving an appropriate definition of singular stratification (adapted to the case of $p$-minimizing maps), we exploit the very same version of Reifenberg Theorem developed in \cite[Theorems 3.3 and 3.4]{Naber2017} to build a controlled covering of each stratum. Notice that analogous results (still exploiting this technique) are available for approximate harmonic maps \cite{Naber2016} and for $Q$-valued energy minimizers \cite{Hirsch2017}.

\subsection{Definitions and notation}
From now on, $\varn$ will always be an $n$-dimensional compact, smooth Riemannian manifold without boundary, while $\mM$ is a $m$-dimensional smooth Riemannian manifold. For convenience, we can assume it is isometrically embedded in a Euclidean space $\R^N$ -- thanks to the well known Nash embedding Theorem (see \cite{Nash1954}). For any $p\in\pths*{1,\infty}$, one can then define the Sobolev space $\wonepN*{\varm}$ as the space of maps $u\in\wonep\pths{\varm,\R^N}$ such that $u(x)\in\varn$ for almost every $x\in\varm$. 

\begin{remark}[Relation between $p$ and $m$]
If $p>m$, by the Sobolev embedding theorem one has that $\wonepN{\varm}\hookrightarrow\C^0(\varm,\varn)$. Consequently, this case is not particularly interesting for our purposes, since the singular set is empty. More generally, in the case $\floor{p}\geq m$ (thus also when $p$ is an integer and $p=m$), it is a known fact that $p$-minimizers have no singular points (see for example \cite[Theorem 2.19]{Naber2014}). \emph{Thus, from now on, we will implicitly assume $p<m$.}
\end{remark}
 
The aim of this paper is to study the singular set of minimizers of the $p$-energy. We now give the main definitions involved. 

\begin{definition}[$p$-energy]
If $u\in W^{1,p}\pths*{\varm,\varn}$, with $\pths{\varm,g}$ Riemannian manifold of dimension $m$, and $\varn\hookrightarrow \R^N$, then we define the \textbf{$p$-energy} of $u$ as
	\begin{equation}
		\ene_p(u)\doteq\int_{\varm}\abs*{\nabla u(x)}_{\varn}^p\,\omega_g,
	\end{equation}
where $\omega_g$ is the volume form associated to the metric $g$. Recall that, in a local coordinate chart, we have: 
	\begin{align}
		\abs*{\nabla u(x)}&=\sqrt{g^{ij}(x)\ango*{\pdrvarg{x_i}{u}(x),\pdrvarg{x_j}{u}(x)}_{\varn}}	\\
		\omega_g&=\sqrt{\det g}\,dx_1\cdots dx_m,
	\end{align}
where the scalar product $\ango{\cdot,\cdot}_{\varn}$ is the Riemannian scalar product in the target $\varn$.
\end{definition}

\begin{definition}[Singular set of a map]
If $u\in W^{1,p}\pths*{\varm,\varn}$, we define its \textbf{singular set} as
	\begin{equation}
		\singu\doteq \brcs*{x\in\varm\stset \text{$u$ is not continuous at $x$}}.
	\end{equation}
\end{definition}

\paragraph{Assumptions on the domain.}\label{paragraph: assumptions on omega}
Since the problem of studying $\singu$ can be treated as a local problem, from now on we assume the setting to be the following: $u$ is defined on an open domain $\Omega$ of $\R^m$; $\Omega$ contains a ball $\ball{0}{\bar{R}}$ for some $\bar{R}\geq 2$; and we study the intersection $\singu\cap\ball{0}{1}$. We won't be more precise than this about the actual size we need for $\ball{0}{\bar{R}}$: this will be the result of a collection of assumptions we will gradually make in the next sections. Notice that, given this assumption, the $p$ energy $\ene_p(u)$ actually assumes the simpler form \begin{equation}
		\ene_p(u)=\int_\Omega \pths*{\sum_{i=1}^m\sum_{\alpha=1}^N\pths*{\pdrvarg{x_i}{u^\alpha}(x)}^2}^{\half[p]}\,dx
	\end{equation}
for any map $u\in\wonepN*{\Omega}$.
\begin{remark}\label{remark: curvature}
When the domain is a generic smooth Riemannian manifold $\pths{\varm,g}$ of dimension $m$, the problem can be reduced to the Euclidean setting by suitably choosing a coordinate chart around a given point (see \cite[Section 7]{Hardt1987}). For any point $q\in\varm$, a system of \emph{normal} coordinates is defined on a geodesic ball $B^{\varm}_\epsilon(q)$ around $q$; in these coordinates, the metric $g$ is represented by $g_{ij}(q)=\delta_{ij}$ at the point $q$. By smoothness, up to lowering the radius $\epsilon$, we can assume that $\abs*{g-\Id}$ is small in $B^{\varm}_\epsilon(q)$. Under this assumption, one can give lower and upper bounds on the $p$-energy
	\begin{equation}
		\ene_{p,g}(u)=\int_{B^{\varm}_\epsilon(q)}\abs*{\nabla u}_g^p\sqrt{\det g}\,dx
	\end{equation}
in terms of 
	\begin{equation}
		\ene_{p,\text{flat}}(u\circ\phi^{-1})=\int_{\ball{0}{\epsilon}}\abs*{\nabla\pths*{ u\circ\phi^{-1}}}^p\,dx,
	\end{equation}
where $\phi:B^{\varm}_\epsilon(q)\to\ball{0}{\epsilon}\subset\R^m$ is the coordinate chart.
Through this procedure, one can show that suitable versions of the main results we use ($\epsilon$-regularity, monotonicity formula) still hold (locally) in the case of non-flat manifolds (details and explicit computations can be found in \cite[Section 7]{Hardt1987} and \cite[Section 2.2]{Xin1996}). Notice that the quantitative results we prove will now have constants depending on the injectivity radius of $\varm$, due  to the fact that we needed to choose normal coordinates on a geodesic ball.
\end{remark}
We are interested in studying \emph{minimizers} of the $p$-energy, \ie{} maps which minimize the functional $\ene_p$ among those with the same boundary datum. More precisely, we give the following definitions:

\begin{definition}
Let $u\in\wonepN*{\Omega}$. We say that:
	\begin{enumerate}
		\item $u$ is a \textbf{$p$-energy minimizing} map if $\ene_p(u)\leq\ene_p(v)$ for any compact set $K\subset\Omega$ and for any $v\in\wonepN*{\Omega}$ \st{} $\restr{u}{\Omega\setminus K}\equiv\restr{v}{\Omega\setminus K}$ (more precisely: $u=v$ almost everywhere in $\Omega\setminus K$).
		\item $u$ is a \textbf{weakly $p$-harmonic} map if it is a critical point of the $p$-energy functional with respect to variations in the target manifold, \ie{}: for any $\xi\in\C^\infty_c(\Omega,\R^N)$, it holds:
			\begin{equation}
				\restr{\drv{t} \int_{\Omega} \abs*{\nabla \pths*{\Pi_{\varn}(u+t\xi)}}^p}{t=0}=0,
			\end{equation}
		where $\Pi_{\varn}$ is the nearest-point projection onto $\varn$, defined on a tubular neighborhood of $\varn$ itself.
		\item $u$ is a \textbf{stationary $p$-harmonic} map if it is weakly $p$-harmonic and it is a critical point of the $p$-energy functional with respect to compact variations in the domain. Explicitly: let $\Phi=\brcs*{\phi_t}_{t\in I}$ be any smooth family of diffeomorphisms of $\Omega$, with $I$ open interval containing $0$; assume that $\phi_0\equiv\id_\Omega$, and that there exists a compact set $K\subset\Omega$ \st{} $\restr{\phi_t}{\Omega\setminus K}=\id_{\Omega\setminus K}$ for any $t\in I$; then
			\begin{equation}
				\restr{\drv{t} \ene_p \pths*{u\circ\phi_t}}{t=0}= \drv{t} \restr{\int_{\Omega}\abs*{\nabla(u\circ\phi_t)(x)}^p dx}{t=0}=0.
			\end{equation}
	\end{enumerate}
\end{definition}
It is clear that any $p$-energy minimizing map is (both weakly $p$-harmonic and) stationary $p$-harmonic. The results we achieve in this article will always be proved for $p$-energy minimizers. We now briefly survey the standard and classical tools available when dealing with $p$-minimizers: these techniques will be then expanded in \cref{section: preliminaries,section: qs}.

\subsection{Basic tools and references: an overview}\label{subsec: references}
Some essential references to keep in mind will be the work of Schoen and Uhlenbeck \cite{Schoen1982}, for fundamental results regarding $2$-harmonic maps, and its generalization to generic $p$, due to Hardt and Lin \cite{Hardt1987}. 
\paragraph{Monotonicity formula.} Given a map $u\in\wonepN{\Omega}$ and a ball $\ball{x}{r}\subset\Omega$, in \cref{definition: normalized energy} we define a quantity $\theta(x,r)$ which will turn out to be \emph{scale invariant} and \emph{monotone in $r$}. This takes the same form of the normalized energy introduced in \cite{Schoen1982} for $2$-harmonic maps and in \cite{Hardt1987} for $p$-harmonic maps (see also \cite[Section 2.4]{Simon1996}): indeed, we'll consider the function
	\begin{equation}
		\theta(x,r)=r^{p-m}\int_{\ball{x}{r}}\abs*{\nabla u}^p\,dx.
	\end{equation}
In \cref{theorem: monotonicity formula} we'll show that an explicit formula for $\frac{d}{dr}\theta(x,r)$ is available. As a consequence, it will be clear that $\theta(x,\cdot)$ is non-decreasing, and that $\theta(x,r)=\theta(x,s)$ with $r> s$ \iff{} $u$ is $0$-homogeneous in a ball of controlled radius centered at $x$. Making this statement \emph{quantitative} will be a key idea: 
\begin{quote}\itshape
\textbf{Heuristic Principle:} if the difference $\theta(x,r)-\theta(x,s)$ is small enough, then $u$ will be close to be $0$-homogeneous in $\ball{x}{cr}$; if this happens at $k+1$ points sufficiently far away from each other, then $u$ will be close to be invariant along a $k$-plane in a ball $\ball{x_0}{cr}$.
\end{quote}
We will turn this principle in a precise statement in \cref{corollary: quantitative rigidity}; other version of the same idea will be used throughout the article.

\paragraph{Quantitative stratification.} Based on ideas contained in \cite{Schoen1982}, the following notion of \emph{singular stratification} for a $\wonepN*{\Omega}$ map can be introduced:
	\begin{equation}
		\mS^k(u)\doteq \brcs*{x\in\Omega\stset \text{any tangent map to $u$ at $x$ is at most $k$-symmetric}};
	\end{equation}
here a tangent map is an $L^p$ limit of blow ups, and $k$-symmetric means $0$-homogeneous and invariant along a $k$-plane. For $p=2$, it was proved in \cite{Schoen1982} that $\dim_{\haus}\!\mS^k(u)\leq k$. In \cref{definition: quantitative stratification} we give a \emph{quantitative} version of this definition, based on the ones appearing in \cite{Cheeger2013} for $p=2$ and \cite{Naber2014} for $p$ generic: given three parameters $k$, $\eta$ and $r$ we look at points $x$ \st{} \inquotes{in $\ball{x}{r}$, $u$ is \emph{not} $\eta$-close to be $k$-invariant}. We'll say that any such point belongs to the stratum $\strat[u]{k}{\eta}{r}$. Another way to read the Heuristic Principle is that:
{\itshape if $x_0\in\strat{k}{\eta}{r}$, then the points in $\ball{x_0}{cr}$ for which $\theta(x,r)-\theta(x,s)$ is small must lie close to a $k$-plane.}

\paragraph{$\epsilon$-regularity.} The link between the study of singular sets ond the study of singular stratifications is an $\epsilon$-regularity result (again from \cite{Schoen1982,Hardt1987}): if $\theta(x,r)$ is sufficiently small, then $u$ is regular in $\ball{x}{\half[r]}$. With a simple compactness argument, one can generalize this result to prove that for some $\eta$ we have $\singu\subset\strat{m-\floor{p}-1}{\eta}{r}$ for any $r$ (see \cref{prop: singular set}): thus the study of $\singu$ reduces to the study of the $(m-\floor{p}-1)^{\text{th}}$ singular stratum.

\subsection*{Acknowledgements}
I am deeply grateful to Daniele Valtorta for the precious suggestions he gave me while writing this article; and before that, for sharing his knowledge and ideas about this topic.

\section{Preliminaries}\label{section: preliminaries}
As in the case $p=2$, $p$-harmonic maps satisfy (in a weak sense) some suitable Euler-Lagrange equations, which are stated in the next two theorems (see for example \cite[Sections 3.1 to 3.3]{Moser2005} for a complete treatment of the case $p=2$; the same computations work for the case $p\in\pths{1,\infty}$).
\begin{theorem}[External variations] If $u$ is {weakly $p$-harmonic}, then it satisfies (weakly) the equation
	\begin{equation}
	\label{equation: external variation Euler Lagrange}
	\tag{ExtEL}
	\laplace_p u \doteq \diverg\pths*{\abs*{\nabla u}^{p-2}\nabla u}=-\abs*{\nabla u}^{p-2}A(u)\pths*{\nabla u, \nabla u},
	\end{equation}
	where $A$ is the second fundamental form of the embedding $\varn\hookrightarrow\R^N$. Explicitly, we have
	\begin{equation}
	\int_\Omega \abs{\nabla u}^p\ango*{\nabla u,\nabla \phi}\,dx=-\int_\Omega \abs{\nabla u}^p A(u)\pths*{\nabla u, \nabla u}\phi\,dx,
	\end{equation}
	for any $\phi\in\C^\infty_c(\Omega,\R^N)$.
\end{theorem}

\begin{example}
It is easy to show that the projection on the unit sphere $u:B^N_1(0)\to\sphere{N-1}$ defined by $u(x)=\frac{x}{\abs{x}}$ is in $\wonep\pths*{B^N_1(0),\sphere{N-1}}$ whenever $p<N$, and in that case it satisfies \cref{equation: external variation Euler Lagrange} classically out of the origin; as a consequence, one can immediately show that $u$ is weakly $p$-harmonic in $B^N_1(0)$. A key fact is that the second fundamental form of the sphere can be explicitly computed:
	\begin{equation}
	A(x)\pths*{X(x),Y(x)}=\ango*{X(x),Y(x)}x\qquad \forall x\in\sphere{N-1}.
	\end{equation}
It is then an easy exercise to show that the equation actually holds; both sides of the equation turn out to be equal to $-(N-1)^{\half[p]}\abs{x}^{-1-p}x$.
\end{example}

\begin{theorem}[Internal variations] If $u$ is {stationary $p$-harmonic}, then for any $X\in \C^\infty_c(\Omega,\R^m)$
	\begin{equation}
	\label{equation: internal variation Euler Lagrange}
	\tag{IntEL}
		\int_{\Omega}\abs*{\nabla u}^{p-2}\sum_{i,k=1}^m\big[p\langle\nabla_i u,\nabla_k u\rangle - \abs{\nabla u}^2\delta_{ik}\big]\pdrvarg{x^i}{X^k} dx=0.
	\end{equation}
\end{theorem}

A crucial tool in the study of ($p$-)harmonic maps is the \emph{normalized energy} of a map in a fixed ball of the domain. 
\begin{definition}[Normalized energy]\label{definition: normalized energy}
		Let $u$ be a $\wonepN*{\Omega}$ map. Let $\psi\in \C^{\infty}_c([0,\infty))$ be a non-increasing function supported in $[0,\bar{R})$. For all $x\in\ball{0}{1}$ and $r>0$ small enough, we define the \textbf{normalized $p$-energy} as the function
			\[\theta(x,r)=\theta_\psi[u](x,r)\doteq r^{p-m}\int_{\Omega}\psi\pths*{\frac{\abs{y-x}}{r}}\abs{\nabla u (y)}^p\, dy.\]
	\end{definition}
The definition we give here is actually not the standard one (although it has already been used in several articles on the argument): one recovers the usual definition (present for example in \cite{Schoen1982,Cheeger2013,Naber2017}) by taking $\psi=\xi_{\bkts{0,1}}$ (which is not an admissible choice in our definition). 
\begin{remark}[Motivation for the definition]
As we will see in \cref{theorem: monotonicity formula} (and all the subsequent results), a condition of the type $\theta_\psi(x,r)-\theta_\psi(x,s)=0$ with $r>s$ gives much more information on $u$ with our definition, rather than with the standard one: for example, one can only deduce that $u$ is $0$-homogeneous in the annulus $\brcs*{y: s<\abs{y-x}<r}$ if $\psi=\xi_{\bkts{0,1}}$, while we obtain $0$-homogeneity in a whole ball around $x$ if we choose $\psi$ such that $\psi'<0$ in an interval $(0,\bar{t})$. 
\end{remark}
A first useful property of the normalized energy is that it is scale invariant:
	
\begin{definition}[Blow-ups]
	Let $u:\Omega\to\varn$, and let $\ball{x}{r}\subset\Omega$. We define the \textbf{blow-up} of $u$ (centered at $x$, with scale $r$) as the map 
	\begin{equation}
		T_{x,r}u(y)\doteq u(x+ry);
	\end{equation}
	the definition makes sense on the set
	\begin{equation}
		\frac{\Omega-x}{r}\doteq\brcs*{y\in\R^m\stset x+ry\in\Omega}\supset\ball{0}{1}.
	\end{equation}
\end{definition}

\begin{theorem}[Scale invariance]
	If $u\in \wonepN*{\Omega}$, $x\in\ball{0}{1}$ and $r>0$, the following identity holds:
	\begin{equation}
		\theta_\psi\bkts*{T_{x,r}u}(0,1)=\theta_\psi\bkts*{u}(x,r).
	\end{equation}
	As a further consequence, if also $w\in \frac{\Omega-x}{r}$ and $\tau>0$ is small enough,
	\begin{equation}
		\theta_\psi\bkts*{T_{x,r}u}(w,\tau)=\theta_\psi\bkts*{u}(x+rw,r\tau).
	\end{equation}
\end{theorem}

\begin{theorem}[Monotonicity formula]\label{theorem: monotonicity formula}
Let $u$ be a stationary $p$-harmonic map, and $\psi$ a smooth function as before. Fix $x\in\ball{0}{1}$ and $r>0$ smaller than $\dist\pths{x,\partial\Omega}$. Then $\theta_\psi(x,\cdot)$ has a derivative at $r$ and the following equality holds:
	\begin{equation}\label{equation: monotonicity formula}
	\drv{r}\theta_\psi (x,r) = -pr^{p-m-2}\int_{\Omega}\abs{y-x} \psi'\pths*{\frac{\abs{y-x}}{r}}\abs{\nabla u (y)}^{p-2}\abs{\partial_{r_x(y)}u(y)}^2dy. \tag{MF}
	\end{equation}
\end{theorem}

\begin{remark}
As we already mentioned in the remark at page \pageref{remark: curvature}, this statement is not true if the domain has a non-zero curvature: however, up to transforming the domain as we did in the aforesaid remark, it is indeed true that $e^{Cr}\theta(x,r)$ is monotone, with $C$ only depending on $m$ and $p$ (see \cite[Section 7]{Hardt1987} and \cite[Theorem 2.7]{Xin1996}). As it is easily seen, this modification does not affect our computations.
\end{remark}

\begin{proof} We'll proceed in two steps.
\begin{stepizec}{Step}
\item\label{step:base} We first consider the case $x=0$, $r=1$; the general case will then follow by scale invariance. In particular, we have to prove the following identity:
	\begin{equation}\label{equation: mf at the origin}
		\restr{\drv{r}\theta_\psi (0,r)}{r=1} = -p\int_{\Omega}\abs{y} \psi'\pths*{\abs{y}}\abs{\nabla u (y)}^{p-2}\abs{\partial_{\frac{y}{\abs{y}}}u(y)}^2dy.
	\end{equation}
The key idea is to find a suitable vector field to plug into the Euler-Lagrange equation \eqref{equation: internal variation Euler Lagrange}: thus, we consider the following one:
\[Y(y)=\psi\pths{\abs{y}} y\in \C^\infty_c(\ball{0}{\bar{R}},\R^m).\]
A simple computation gives, for $1\leq i,j\leq m$,
\[\pdrvarg{y^i}{Y^j}=\psi'\pths{\abs{y}}\frac{y_i y_j}{\abs{y}}+\psi\pths{\abs{y}}\delta_{ij}.\]
Then, with this choice of $Y$, the integral appearing in \cref{equation: internal variation Euler Lagrange} reads:
	\begin{equation}\label{equation: mf proof}
		\int_{\Omega}\abs{\nabla u}^{p-2}\bkts*{p\abs{y}\psi'(\abs{y})\abs*{\partial_{\frac{y}{\abs*{y}}}u}^2-\abs{y}\psi'(\abs{y})\abs{\nabla u}^2+(p-m)\psi(\abs{y})\abs{\nabla u}^2}dy;
	\end{equation}
this follows by a straightforward computation, and by the fact that:
	\begin{equation}
		\sum_{i,j=1}^m \ango*{y_i\nabla_i u,y_j\nabla_j u}=\abs{y}^2\ango*{\sum_{i=1}^m\frac{y_i}{\abs{y}}\nabla_i u, \sum_{j=1}^m\frac{y_j}{\abs{y}}\nabla_j u} =\abs{y}^2\abs*{\partial_{\frac{y}{\abs{y}}}u(y)}^2.
	\end{equation}
Now by \cref{equation: internal variation Euler Lagrange} the integral in \eqref{equation: mf proof} is zero; hence \cref{equation: mf at the origin} follows easily, just by taking the derivative of $\theta_\psi(0,\cdot)$ at $r=1$ (and changing the order of integral and derivative): 
\begin{multline}
	\drv{r}\theta_\psi (0,r)=(p-m)r^{p-m-1}\int_{\Omega}\psi\pths*{\frac{\abs{y}}{r}}\abs{\nabla u (y)}^p\, dy+\\+r^{p-m}\int_{\Omega}\psi'\pths*{\frac{\abs{y}}{r}}\pths*{-\frac{\abs{y}}{r^2}}\abs{\nabla u}^p\,dy.
\end{multline}
\item Consider now the general case: arbitrarily fix $x\in\ball{0}{1}$ and $\bar{r}>0$. By scale invariance, we know that $\theta_\psi[u](x,r)=\theta_\psi\bkts*{T_{x,r}u}(0,1)$ for all $r$ in a neighborhood of $\bar{r}$. Hence in particular
\begin{equation}
\restr{\drv{r}\theta_\psi[u](x,r)}{r=\bar{r}}=\restr{\drv{r}\theta_\psi\bkts*{T_{x,r}u}(0,1)}{r=\bar{r}}.
\end{equation} 
Notice that by \ref{step:base} we have information about the quantity $\drv{s}\theta_\psi\bkts*{T_{x,\bar{r}}u}(0,s)$ at $s=1$, which is not directly the information we seek, but is really close. Indeed, a simple computation (which involves nothing more than the definition of $T_{x,r}$) shows that the two quantities are related by
	\begin{equation}
		\restr{\drv{s}\theta_\psi\bkts*{T_{x,\bar{r}}u}(0,s)}{s=1}= \bar{r}  \restr{\drv{r}\theta_\psi\bkts*{T_{x,r}u}(0,1)}{r=\bar{r}}.
	\end{equation}
Thus we have:
	\begin{equation}
		\begin{split}
			\drv{r}\theta_\psi(x,\bar{r})&=\frac{1}{\bar{r}}\restr{\drv{s}\theta_\psi\bkts*{T_{x,\bar{r}}u}(0,s)}{s=1}\\
			&=-\frac{p}{\bar{r}}\int_{\ball{0}{\bar{R}}}\abs{y} \psi'\pths*{\abs{y}}\abs*{\nabla T_{x,\bar{r}}u (y)}^{p-2}\abs*{\partial_{\frac{y}{\abs{y}}} T_{x,\bar{r}}u(y)}^2dy\\
			&=-p\bar{r}^{p-1}\int_{\ball{0}{\bar{R}}}\abs{y} \psi'\pths*{\abs{y}}\abs*{\nabla u (x+\bar{r}y)}^{p-2}\abs*{\partial_{\frac{y}{\abs{y}}} u(x+\bar{r}y)}^2dy.
		\end{split}
	\end{equation}
By performing the change of variables $w=x+\bar{r}y$, we obtain exactly the desired result.
\qedhere
\end{stepizec}
\end{proof}

\begin{corollary}
Let $\psi$ be a smooth function as before; define:
	\begin{equation}
		\Psi(t)\doteq \int_0^t \tau^{p-m}\psi'(\tau)d\tau.
	\end{equation}
Then, for any $p$-stationary map $u$, for any $x\in\ball{0}{1}$ and $0<s<r<\dist\pths{x,\partial\Omega}$, we have:
	\begin{equation}
		\theta_\psi(x,r)-\theta_\psi(x,s)=p\int_{\Omega}\pths*{\Psi\pths*{\frac{\abs{y-x}}{r}}-\Psi\pths*{\frac{\abs{y-x}}{s}}}\abs{y-x}^{m-p}\abs{\nabla u }^{p-2}\abs{\partial_{r_x}u}^2dy.
	\end{equation}
\end{corollary}

\begin{definition}[Assumptions on $\psi$]\label{definition: psi}
From now on, we'll think of $\psi\in \C^{\infty}_c([0,\infty))$ as a fixed function, satisfying 
	\begin{gather}
		\supp\pths*{\psi}=\bkts*{0,t_b},\\
		\psi'(t)<0\quad \text{in $[0,t_b)$},\qquad
		\psi'(t)\leq-\xi \quad \text{in $[0,t_a)$}
	\end{gather}
for some fixed numbers $0<t_a<t_b$ and $\xi>0$ (see \cref{figure: psi}). Moreover, since this will be needed in \cref{section: covering}, we'll actually assume $2<t_a<t_b$: this choice will be better explained in the remark at page \pageref{remark: assumptions on psi}. The function in \cref{figure: psi} can be thought as a valid one.
\end{definition}
Such a choice of $\psi$ is mostly justified by computational reasons; it's worth noting, however, that enlarging the value of $t_b$ is heuristically equivalent to looking at smaller balls in the domain, thus exploiting again the local nature of the problem.

\begin{figure}
\centering
	\begin{tikzpicture}
	\begin{axis}
	[xlabel=$x$, ylabel=$y$, xmin=0,xmax=4.5,ymax=1.2,axis lines = middle, width=7cm, height = 4cm]
		\addplot[mark=none, smooth, ultra thick, blue!60!black] table[x index={0}, y index={1}, col sep = comma]{
				0.0,1.0
				0.02,0.9966
				0.04,0.9932
				0.06,0.9898
				0.08,0.9863
				0.1,0.9827
				0.12,0.9791
				0.14,0.9755
				0.16,0.9718
				0.18,0.968
				0.2,0.9642
				0.22,0.9604
				0.24,0.9565
				0.26,0.9526
				0.28,0.9487
				0.3,0.9446
				0.32,0.9406
				0.34,0.9365
				0.36,0.9324
				0.38,0.9282
				0.4,0.924
				0.42,0.9197
				0.44,0.9154
				0.46,0.911
				0.48,0.9067
				0.5,0.9022
				0.52,0.8978
				0.54,0.8933
				0.56,0.8887
				0.58,0.8841
				0.6,0.8795
				0.62,0.8749
				0.64,0.8702
				0.66,0.8654
				0.68,0.8607
				0.7,0.8559
				0.72,0.851
				0.74,0.8461
				0.76,0.8412
				0.78,0.8363
				0.8,0.8313
				0.82,0.8263
				0.84,0.8212
				0.86,0.8161
				0.88,0.811
				0.9,0.8059
				0.92,0.8007
				0.94,0.7955
				0.96,0.7902
				0.98,0.7849
				1.0,0.7796
				1.02,0.7742
				1.04,0.7689
				1.06,0.7634
				1.08,0.758
				1.1,0.7525
				1.12,0.747
				1.14,0.7415
				1.16,0.7359
				1.18,0.7303
				1.2,0.7247
				1.22,0.719
				1.24,0.7133
				1.26,0.7076
				1.28,0.7019
				1.3,0.6961
				1.32,0.6903
				1.34,0.6844
				1.36,0.6786
				1.38,0.6727
				1.4,0.6668
				1.42,0.6608
				1.44,0.6549
				1.46,0.6489
				1.48,0.6428
				1.5,0.6368
				1.52,0.6307
				1.54,0.6246
				1.56,0.6185
				1.58,0.6123
				1.6,0.6061
				1.62,0.5999
				1.64,0.5937
				1.66,0.5875
				1.68,0.5812
				1.7,0.5749
				1.72,0.5686
				1.74,0.5622
				1.76,0.5558
				1.78,0.5494
				1.8,0.543
				1.82,0.5366
				1.84,0.5301
				1.86,0.5236
				1.88,0.5171
				1.9,0.5106
				1.92,0.5041
				1.94,0.4975
				1.96,0.4909
				1.98,0.4843
				2.0,0.4777
				2.02,0.471
				2.04,0.4644
				2.06,0.4577
				2.08,0.451
				2.1,0.4443
				2.12,0.4376
				2.14,0.4308
				2.16,0.424
				2.18,0.4173
				2.2,0.4105
				2.22,0.4037
				2.24,0.3968
				2.26,0.39
				2.28,0.3832
				2.3,0.3763
				2.32,0.3694
				2.34,0.3626
				2.36,0.3557
				2.38,0.3488
				2.4,0.3419
				2.42,0.335
				2.44,0.328
				2.46,0.3211
				2.48,0.3142
				2.5,0.3073
				2.52,0.3003
				2.54,0.2934
				2.56,0.2865
				2.58,0.2796
				2.6,0.2726
				2.62,0.2657
				2.64,0.2588
				2.66,0.2519
				2.68,0.245
				2.7,0.2382
				2.72,0.2313
				2.74,0.2245
				2.76,0.2176
				2.78,0.2108
				2.8,0.204
				2.82,0.1973
				2.84,0.1906
				2.86,0.1839
				2.88,0.1772
				2.9,0.1706
				2.92,0.164
				2.94,0.1575
				2.96,0.151
				2.98,0.1446
				3.0,0.1382
				3.02,0.1319
				3.04,0.1257
				3.06,0.1195
				3.08,0.1134
				3.1,0.1074
				3.12,0.1015
				3.14,0.0957
				3.16,0.09
				3.18,0.0844
				3.2,0.0789
				3.22,0.0735
				3.24,0.0683
				3.26,0.0632
				3.28,0.0582
				3.3,0.0534
				3.32,0.0488
				3.34,0.0443
				3.36,0.04
				3.38,0.036
				3.4,0.0321
				3.42,0.0284
				3.44,0.0249
				3.46,0.0217
				3.48,0.0187
				3.5,0.0159
				3.52,0.0134
				3.54,0.0111
				3.56,0.009
				3.58,0.0072
				3.6,0.0056
				3.62,0.0043
				3.64,0.0032
				3.66,0.0023
				3.68,0.0016
				3.7,0.001
				3.72,0.0006
				3.74,0.0004
				3.76,0.0002
				3.78,0.0001
				3.8,0.0
				3.82,0.0
				3.84,0.0
				3.86,0.0
				3.88,0.0
				3.9,0.0
				3.92,0.0
				3.94,0.0
				3.96,0.0
				3.98,0.0		
		};
	\end{axis}
	\end{tikzpicture}
\caption{Example for $\psi$, with $t_b=4$, $t_a=3.5$, $\xi=0.1$.}%
\label{figure: psi}
\end{figure}
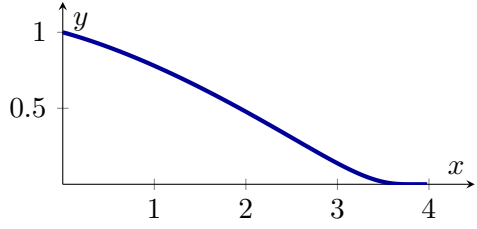

A straightforward consequence of the monotonicity formula is the following:

\begin{corollary}
Let $0<s<r$; let $u\in\sobd{\Omega}$ be a stationary harmonic map, and $x\in\ball{0}{1}$. Take $\psi\in\C^{\infty}_c([0,\infty))$ as in \cref{definition: psi}. If
	\begin{equation}
		\theta\pths*{x,r}-\theta\pths*{x,s}=0,
	\end{equation}
then $u$ is $0$-homogeneous in $\ball{x}{t_br}$ (\wrt{} $x$).
\end{corollary}

Moreover, by simple geometric considerations, if a map is $0$-homogeneous with respect to different points, then it is invariant along the affine subspace generated by those points:
\begin{corollary}[Rigidity]\label{corollary: rigidity}
Let $0<s<r$; let $u\in\wonepN*{\Omega}$ be a stationary harmonic map, and take $\psi\in\C^{\infty}_c([0,\infty))$ as in \cref{definition: psi}. Let $0\leq k\leq m$ be an integer. If there exist $k+1$ points $\brcs*{x_i}_{i=0}^k$ \st{}:
	\begin{itemize}
		\item $x_i\in\ball{x_0}{\half{t_br}}\subset\Omega$ for any $i=1,\dots,k$;
		\item $\brcs*{x_i}_{i=0}^k$ span a $k$-dimensional affine subspace $L$;
		\item For all $i=0,\dots,k$, 
			\begin{equation}
				\theta(x_i,r)-\theta(x_i,s)=0;
			\end{equation}					
	\end{itemize}
then $u$ is $L$-invariant in $\ball{x_0}{\half{t_br}}$, and $0$-homogeneous at any point of $L$.
\end{corollary}

\subsection[An epsilon-regularity result]{An $\epsilon$-regularity result}
Combining Theorems 2.5 and 3.1 of \cite{Hardt1987}, we get the following H\"{o}lder-regularity result for $p$-minimizers:
\begin{theorem}[$\epsilon$-regularity]\label{theorem: epsilon regularity}
There exist two constants $\epsilon_0=\epsilon_0\pths*{m,\varn,\Lambda,p}$ and $\alpha\pths*{m,\varn,\Lambda,p}$ such that the following holds. Let $u\in\wonepN*{\Omega}$ be a minimizer for the $p$-energy, with $\ene_p(u)\leq \Lambda$. If $\theta(x,r)<\epsilon_0$ for some $\ball{x}{r}\subset\Omega$, then $u$ is $\C^{1,\alpha}$-regular in $\ball{x}{\half[r]}$.
\end{theorem}
This result will be the key argument that connects the singular set to the singular stratification (see \cref{prop: singular set}).\par
Notice that the situation gets even better when dealing with 2-harmonic maps: by standard elliptic regularity arguments one can get $\C^\infty$ regularity instead of H\"{o}lder regularity (see \cite{Schoen1982,Schoen1984}). 

\section{Quantitative stratifications}\label{section: qs}
Let $u\in\wonepN*{\Omega}$, and assume $L\in\grass{k}$ is a $k$-linear subspace. We denote by $\abs*{\nabla_L u}$ or $\abs*{\ango{\nabla u, L}}$ the quantity
	\begin{equation}
		\abs*{\nabla_L u} \doteq \pths*{\sum_{i=1}^k\abs*{\ango*{\nabla u,v_i}}^2}^{\half},
	\end{equation}
where $\brcs{v_1,\dots,v_k}$ is any orthonormal basis of $L$.
\begin{definition}[Quantitative stratification]\label{definition: quantitative stratification}
Let $x\in\ball{0}{1}$ and $r>0$ small enough for $\ball{x}{r}$ to be contained in $\Omega$. Fix $k\in\brcs{0,\dots,m}$ and a parameter $\eta>0$. We say that $u$ is \textbf{$(\eta,k)$-invariant in $\ball{x}{r}$} if there exists a linear subspace $L\in\grass{k}$ \st{} 
	\begin{equation}
		r^{p-m}\int_{\ball{x}{r}}\abs{\nabla_L u}^p\,dy<\eta.
	\end{equation}
Equivalently, there exists $L\in\grass{k}$ \st{} 
	\begin{equation}
		\int_{\ball{0}{1}}\abs*{\nabla_L T_{x,r}u}^p <\eta.
	\end{equation}
When this condition holds for some $\eta$ and $k$, we'll generically refer to it as \inquotes{almost invariance}.
We define the \textbf{singular $k^{\text{th}}$ stratum} of $u$, with \textbf{scale parameter $r$} and \textbf{closeness parameter $\eta$} the subset of $\ball{0}{1}$ defined by
	\begin{equation}
		\begin{split}
		\strat{k}{\eta}{r}&\doteq\brcs*{x\in\ball{0}{1}\stset \text{$u$ is \underline{not} $(\eta,k+1)$-invariant in $\ball{x}{s}$ for any $s\geq r$}}=\\
		&=\brcs*{x\in\ball{0}{1}\stset s^{p-m}\!\!\int_{\ball{x}{s}}\abs{\nabla_L u}^p\,dy\geq\eta \text{ for all $L\in\grass{k+1}$ and $s\geq r$}}.
		\end{split}
	\end{equation}
Finally, we denote by $\mS^k_\eta(u)$ the intersection
	\begin{equation}
	\mS^k_\eta(u)\doteq\bigcap_{r>0}\strat{k}{\eta}{r}=\brcs*{x\in\ball{0}{1}\stset \text{$u$ is \underline{not} $(\eta,k+1)$-invariant in $\ball{x}{s}$ for any $s>0$}},
	\end{equation}
for any given $\eta$ and $k$.
\end{definition}

\begin{remark}
Notice that this definition is slightly different from the one used in \cite{Naber2014} and originally introduced in \cite{Cheeger2013}: indeed, in those papers $u$ is defined to be \inquotes{almost $k$-invariant} if it is (quantitatively) close in $L^p$ to a map which is $0$-homogeneous and $k$-invariant. The two approaches can be shown to be equivalent (see for example \cite[Proposition 3.11]{Hirsch2017}); we won't explore this path further.
\end{remark}

In the following proposition we show that the definition of almost invariance given above implies an \inquotes{almost 0-homogeneity} condition at a smaller scale (and this is why we didn't require such a condition in the very definition of almost invariance). Not only: given a ball $\ball{x}{r}$ where $u$ is $(\delta,k)$-invariant for a suitable $\delta$, all the points in $\ball{x}{\half r}$  are both almost invariant and almost 0-homogeneous \wrt{} this smaller scale.

\begin{proposition} \label{proposition: almost homogeneity}
Let $\eta>0$ be a fixed parameter. There exists a constant $\bar{\gamma}(m,\varn,\Lambda,p,\eta)$, with $0<\bar{\gamma}<\half$, such that the following holds. Define $\bar{\delta}\doteq \bar{\gamma}^{2(m-p)}$. Let $u$ be a $p$-minimizing map with $\ene_p(u)\leq\Lambda$, and $\ball{x}{r}\subset\ball{0}{1}$. Assume $u$ is $(\bar{\delta},k)$-invariant in $\ball{x}{r}$:
	\begin{equation}
		r^{p-m}\int_{\ball{x}{r}}\abs*{\nabla_L u}^p\,dz<\bar{\delta}
	\end{equation}
for some $L\in\grass{k}$. Then for any $y\in\ball{x}{\half r}$ there exists a radius $\bar{\gamma} r\leq r_y\leq\half r$ such that $u$ satisfies the following \emph{almost invariance} and \emph{almost 0-homogeneity} conditions in $\ball{y}{r_y}$:
	\begin{gather}
		r_y^{p-m}\int_{\ball{y}{r_y}}\abs*{\nabla_L u}^p\,dz<\eta\\
		\theta(y,r_y)-\theta\pths*{y,\half r_y}<\eta.\label{equation: almost 0 homogeneity}
	\end{gather}
\end{proposition}

\begin{proof}
By scale invariance, it is sufficient to prove the statement for $x=0$, $r=1$. Choose $\bar{\gamma}$ so that
	\begin{equation}
		\bar{\gamma}<\min\brcs*{2^{-\frac{\Lambda}{\eta}-1},\eta^{\frac{1}{m-p}},\half}.	\label{equation: bargamma definition}
	\end{equation}
Consider a point $y\in\ball{0}{\half}$; assume, by contradiction, that for all $i$ such that $2^{-i}\geq \bar{\gamma}$ we have
	\begin{equation}
		\theta\pths*{y,2^{-i}}-\theta\pths*{y,2^{-i-1}}\geq \eta.
	\end{equation}
Then we should have:
	\begin{equation}
		\Lambda>\theta\pths*{y,\half}
		\geq \sum_{i=1}^{\floor*{\frac{\Lambda}{\eta}+1}}\pths*{\theta\pths*{y,2^{-i}}-\theta\pths*{y,2^{-i-1}}}
		\geq \floor*{\frac{\Lambda}{\eta}+1}\eta> \Lambda,
	\end{equation}
a contradiction. Thus for any $y\in\ball{0}{\half}$ we have a radius $r_y\in\bkts*{\bar{\gamma},\half}$ for which \eqref{equation: almost 0 homogeneity} holds.
Moreover, by \eqref{equation: bargamma definition} we also have:
	\begin{equation}
		r_y^{p-m}\int_{\ball{y}{r_y}}\abs*{\nabla_L u}^p\,dz\leq \bar{\gamma}^{p-m}\bar{\delta}=\bar{\gamma}^{m-p}<\eta,
	\end{equation}
which concludes the proof.
\end{proof}

The bridge between singular stratification and singular set of a $p$-harmonic map is given by the following \lcnamecref{prop: singular set}, which strongly relies on the $\epsilon$-regularity \cref{theorem: epsilon regularity}. This result can be seen as a quantitative version of the following known fact (\cite[Theorem 4.5]{Hardt1987}): a $p$-minimizing map which is $0$-homogeneous and invariant along a $(m-\floor{p})$-linear subspace must be constant.

\begin{proposition}[Singular set and stratification]\label{prop: singular set}
Let $u\in\wonepN*{\Omega}$ be a $p$-energy minimizing map with energy bounded by $\Lambda$. There exists $\eta=\eta(m,\varn,\Lambda,p)$ such that for all $r>0$ (small) we have
	\begin{equation}
		\singu\cap\ball{0}{1}\subset\strat{m-\floor{p}-1}{\eta}{r}.
	\end{equation}
\end{proposition}

\begin{proof}
For any $i\in\N$, let $\gamma_i\doteq\bar{\gamma}\pths*{m,\varn,\Lambda,p,\frac{1}{i}}$ be the constant given by \cref{proposition: almost homogeneity} when $\eta=\frac{1}{i}$, and let $\delta_i\doteq \gamma_i^{2(m-p)}$.\par
Argue by contradiction: assume that for all $i\in\N$ there exists a $p$-minimizing map $u_i$ with $\ene_p(u_i)\leq \Lambda$, a singular point $x_i\in\sing(u_i)$, a $r_i>0$ and a $(m-\floor{p})$-plane $L_i$ such that
	\begin{equation}
		\int_{\ball{0}{1}}\abs*{\nabla_{L_i} T_{x_i,r_i}u_i}^p<\delta_i.
	\end{equation}
Up to precomposing with a rotation of the space, we can assume $L_i=L$ for all $i$, for some affine subspace $L$. By \cref{proposition: almost homogeneity}, we have
	\begin{gather}
		(\alpha_i r_i)^{p-m}\int_{\ball{x_i}{\alpha_i r_i}}\abs*{\nabla_L u_i}^p\,dz<\frac{1}{i}\label{equation: singu and stratu 1}\\
		\theta[u_i](x_i,\alpha_i r_i)-\theta\bkts{u_i}\pths*{x_i,\half \alpha_i r_i}<\frac{1}{i}\label{equation: singu and stratu 2}
	\end{gather}
for a sequence $\brcs*{\alpha_i}_i$, with $\gamma_i\leq\alpha_i\leq \half$.
The maps $T_{x_i,\alpha_i r_i}u_i$ are $p$-minimizing, and they are uniformly bounded in $\wonepN*{\ball{0}{1+\epsilon}}$ for some $\epsilon$ (by compactness of $\varn$ and by the bound on the $p$-energy); thus, up to subsequences, they weakly converge to a map $\tilde{u}\in\wonepN*{\ball{0}{1+\epsilon}}$. By \cite{Hardt1987,Luckhaus1988}, the convergence is actually strong in $\wonepN*{\ball{0}{1}}$, and $\tilde{u}$ is a $p$-minimizer in $\ball{0}{1}$. But now by strong $\wonepN*{\ball{0}{1}}$ convergence and by \eqref{equation: singu and stratu 1}, \eqref{equation: singu and stratu 2} we have 
	\begin{gather}
		\int_{\ball{0}{1}}\abs*{\nabla_L \tilde{u}}^p=0,\\
		\theta\bkts{\tilde{u}}\pths*{0,1}-\theta\bkts{\tilde{u}}\pths*{0,\half}=0;
	\end{gather}
so $\tilde{u}$ is $p$-minimizing, $(m-\floor{p})$-invariant on $\ball{0}{1}$ and 0-homogeneous on $\ball{0}{1}$. By \cite[Theorem 4.5]{Hardt1987}, this implies that $\tilde{u}$ is constant on $\ball{0}{1}$: thus in particular $\theta[\tilde{u}](0,\cdot)\equiv 0$ in $\pths{0,1}$. However, by the fact that $x_i\in\sing\pths*{u_i}$, and by the $\epsilon$-regularity \cref{theorem: epsilon regularity}, we have:
	\begin{equation}
		\theta\bkts*{T_{x_i,\alpha_i r_i}u_i}\pths*{0,s}\geq \epsilon_0\qquad\forall s>0,
	\end{equation}
which implies $\theta\bkts{\tilde{u}}\pths*{0,s}\geq\epsilon_0$ by $W^{1,p}$-convergence: we have reached a contradiction.
\end{proof}

Since we'll make great use if compactness and limiting arguments, we will need an effective notion of \inquotes{points in general position} which is preserved when passing to the limit.
\begin{definition}\label{definition: effective span}
Given $k+1$ points $\brcs*{x_i}_{i=0}^k$ in $\R^m$ (with $0\leq k\leq m$), and $\rho>0$, we say that $\brcs{x_i}_i$ are in $\rho$-general position if for all $j=1,\dots,k$
\[\dist\pths*{x_j,x_0+\spann\brcs*{x_1-x_0,\dots,x_{j-1}-x_0}}\geq\rho.\] 
We say that a set of points $\mS$ spans $\rho$-effectively a given $k$-subspace $L\in\aff{k}$ if there exist $k+1$ points $\brcs{x_0,x_1,\dots,x_k}\subset \mS$ in $\rho$-general position \st{}
	\begin{equation}
		L=x_0+\spann\brcs*{x_1-x_0,\dots,x_{k}-x_0}.
	\end{equation}
\end{definition}

As we wanted, the notion of $\rho$-general position passes to the limit:
\begin{lemma}\label{lem:limindep}
For any $j\in\N$, let $\brcs*{x_{ij}}_{i=0}^k$ be $k+1$ points of $\R^m$ in $\rho$-general position, with $\rho>0$. Assume that $x_{ij}\xrightarrow{j\to\infty}\bar{x}_i$ for all $i=0,\dots,k$. Then $\brcs{\bar{x}_i}_{i=0}^k$ are still in $\rho$-general position.
\end{lemma}

We are now ready to state a first precise version of what we called \inquotes{Heuristic Principle} in \cref{subsec: references}. This makes \cref{corollary: rigidity} quantitative; the proof is a simple compactness argument, based on \cref{lem:limindep} and on the (already used) fact that weak $W^{1,p}$ limits of $p$-minimizing maps are actually strong limits, and are $p$-minimizers themselves.

\begin{corollary}[Quantitative rigidity]\label{corollary: quantitative rigidity}
Let $u\in\wonepN*{\Omega}$ be a $p$-minimizing map with energy bounded by $\Lambda$. Let $0\leq k\leq m$ be an integer. Fix the constants $\eta, p, \gamma, \rho>0$. There exists a constant $\epsilon>0$ (depending on $m,\varn,\Lambda,\eta, p, \gamma, \rho$) suich that the following holds. If there exist $k+1$ points $\brcs*{x_i}_{i=0}^k$ \st{}:
	\begin{itemize}
		\item $x_i\in\ball{x_0}{\half{t_br}}\subset\Omega$ for any $i=1,\dots,k$;
		\item $\brcs*{x_i}_{i=0}^k$ span $\rho$-effectively a $k$-dimensional affine subspace $L$;
		\item For all $i=0,\dots,k$, 
			\begin{equation}
				\theta(x_i,r)-\theta(x_i,\gamma r)<\epsilon;
			\end{equation}					
	\end{itemize}
then $r^{p-m}\int_{\ball{x_0}{t_b r}}\abs*{\nabla u}^p<\eta$.
\end{corollary}

\subsection{The results: precise statements}
The main result we will prove is the following.
\hiddennconst{C}{C:finalestimate}
\begin{theorem}[Singular strata]\label{theorem: main}
Let $u\in\wonepN*{\Omega}$ be a $p$-energy minimizing map with energy bounded by $\Lambda$. Let $\eta>0$ and $1\leq k \leq m$.  There exists two constants $\oconst{C}{C:finalestimate}$ and $\delta_0$ depending on $m,\varn,p,\Lambda,\eta$ such that for any $r>0$
	\begin{equation}
	\vol \pths*{\fat{\strat{k}{\eta}{\delta_0r}}{r}\cap\ball{0}{1}}\leq \oconst{C}{C:finalestimate}r^{m-k}.
	\end{equation}
Moreover, for any $\eta>0$ and any $0\leq k\leq m$, the stratum $\mS^k_\eta(u)$ is $k$-rectifiable.
\end{theorem}
The proof will be achieved in \cref{section: proofs}, exploiting all the tools developed in \cref{section:reifenberg,section:awful,section: covering}.
\par
Notice that, thanks to \cref{prop: singular set}, we obtain the following crucial corollary, which is actually the main goal we wanted to achieve: volume estimates and structural information for the singular set.
\hiddennconst{C}{C: singular set}
\begin{corollary}[Singular set]
Let $u\in\wonepN*{\Omega}$ be a $p$-energy minimizing map with energy bounded by $\Lambda$. There exists a constant $\oconst{C}{C: singular set}(m,\varn,\Lambda,p)$ such that for any $r>0$
	\begin{equation}
		\vol \pths*{\fat{\singu}{r}\cap\ball{0}{1}}\leq \oconst{C}{C: singular set}r^{\floor{p}+1}.
	\end{equation}
In particular, the Minkowski (and Hausdorff) dimension of $\singu$ is at most $m-\floor{p}-1$, and the upper Minkowski content is bounded by $\oconst{C}{C: singular set}$.
Furthermore, $\singu$ is $\pths{m-\floor{p}-1}$-rectifiable.
\end{corollary}

\section{Reifenberg Theorems and approximating planes}\label{section:reifenberg}
The next step will be to introduce some more advanced techniques which allow us to analyze the behaviour of each singular stratum \emph{at every scale $r$} around a point $x$. In order to state Reifenberg Theorem (in a form which is suited to our context), we first need to recall the definition of Jones' numbers of a measure $\mu$ (first appeared in \cite{Jones1990}; for a detailed introduction, see \cite{Pajot2002}): this is a scale-invariant notion which quantifies how close $\supp(\mu)$ is to be contained in an affine $k$-space (near a given point).
\begin{definition}[Jones' numbers]
Let $x\in\ball{0}{1}$ and $0<r<1$. Assume $\mu$ is a positive Radon measure on $\Omega$. For any $k\in\brcs{0,\dots,m}$ we define the \textbf{$k$-dimensional Jones' number} of $\mu$ in $\ball{x}{r}$ as
	\begin{equation}
		\beta^k_\mu (x,r)\doteq \pths*{r^{-k}\inf\brcs*{\int_{\ball{x}{r}}\frac{\dist(y,L)^2}{r^2}\,d\mu(y)\stset L\in\aff{k}}}^{\half}.
	\end{equation}	 
\end{definition}

\begin{remark}
The quantity $\beta^k_\mu$ is scale invariant in the following sense. Assume $\mu$ is defined in a ball $\ball{x}{r}$; define the blow up measure $\hat{\mu}=\hat{\mu}^k_{x,r}$ on $\ball{0}{1}$ as
	\begin{equation}
	\hat{\mu}(A)\doteq r^{-k}\mu\pths*{x+rA}\qquad\text{$\forall A\subset\ball{0}{1}$ measurable.}
	\end{equation}
Then it is easy to compute that $\beta^k_{\hat{\mu}}\pths{0,1}=\beta^k_\mu\pths{x,r}$.
\end{remark}

The main hypothesis one needs, in order to apply Reifenberg Theorem in its different forms, is a control on the Jones numbers of a suitable measure (\eg{}, $\haus^k$ restricted to a set). In all the cases, the condition we need takes the following form:
\begin{definition}[Reifenberg condition]
Let $\mu$ be a positive Radon measure on $\Omega$, and $k\in\brcs{0,\dots,m}$. We say that $\mu$ satisfies the ($k$-dimensional) \textbf{Reifenberg condition} with constant $\delta$ if for any $x\in\ball{0}{1}$ and $0<r<1$ we have:
	\begin{equation}
	\label{equation: reifenberg condition}
	\tag{$k$-Reif}
		\int_{\ball{x}{r}}\int_0^r\beta^k_\mu(y,s)^2\frac{ds}{s}d\mu(y)<\delta r^k.
	\end{equation}
\end{definition}

As we will clarify in \cref{theorem: Reifenberg}, two versions of Reifenberg Theorem are available for our purposes: in one of them (necessary for the rectifiability of a set), one needs to check \eqref{equation: reifenberg condition} on the restriction of the Hausdorff measure to the given set; the other one (necessary for volume estimates) makes use of \emph{discrete measures} as the following:
\begin{definition}[Measure associated to a disjoint family of balls]
Assume $\mathcal{C}$ is a (discrete) subset of $\ball{0}{1}$, and $\mathscr{F}=\brcs*{\ball{x}{r_x}}_{x\in\mathcal{C}}$ is a collection of \emph{disjoint} balls centered in $\mathcal{C}$, each contained in $\ball{0}{2}$. For any $k\in\brcs{0,\dots,m}$, we define the following measure associated to $\mathscr{F}$:
	\begin{equation}
		\mu_{\mathscr{F},k}\doteq \sum_{x\in\mathcal{C}} r_x^k \delta_{x},
	\end{equation}
where $\delta_x$ is the Dirac measure centered at $x$.
\end{definition}

\begin{theorem}[Reifenberg]\label{theorem: Reifenberg}
	There exist two constants $C_{\text{R}}$ and $\delta_{\text{R}}$ such that the following statements hold true.
	\begin{enumerate}[(i)]
	\item\label{item:reifenberg one} Assume $\mathscr{F}$ is a family of disjoint balls with centers in $\mathcal{C}\subset{\ball{0}{1}}$, each contained in $\ball{0}{2}$. If $\mu_{\mathscr{F},k}$ satisfies the condition \eqref{equation: reifenberg condition} with constant $\delta_{\text{R}}$, then 
	\begin{equation}
		\sum_{x\in\mathcal{C}}r^k_x	\leq C_{\text{R}}.
	\end{equation}
	\item\label{item:reifenberg two} Assume $S\subset\ball{0}{1}$ is a $\haus^k$-measurable set. If $\haus^k\mres S$ satisfies the condition \eqref{equation: reifenberg condition} with constant $\delta_{\text{R}}$, then $S$ is $k$-rectifiable and
	\begin{equation}
		\haus^k\mres S (\ball{x}{r})\leq C_{\text{R}}r^k
	\end{equation}
	for any $x\in S$ and $0<r<1$.
	\end{enumerate}
\end{theorem}

The original proof of this version of Reifenberg Theorem can be found in \cite[Sections 5 and 6]{Naber2017}; a more general form of it is contained in \cite[Section 2]{Edelen2016}, while similar arguments are developped in \cite{David2012,Edelen2018,Miskiewicz2018,Toro1995,Azzam2015}.
\par Since we will need it in this form, for the sake of clarity we state here a rescaled version of \cref{theorem: Reifenberg}, part \ref{item:reifenberg one}.

\begin{corollary}[Reifenberg, rescaled version]\label{corollary: rescaled reifenberg}
Let $\ball{\bar{x}}{\bar{r}}$ be a fixed ball. Assume $\mathscr{F}$ is a family of disjoint balls with centers in $\mathcal{C}\subset{\ball{\bar{x}}{\bar{r}}}$, each contained in $\ball{\bar{x}}{2\bar{r}}$. If $\mu=\sum_{x\in\mC}r_x^k\delta_x$ satisfies
	\begin{equation}
		\int_{\ball{w}{\bar{r} r}}\int_0^{\bar{r} r}\beta^k_\mu(y,s)^2\frac{ds}{s}d\mu(y)<\delta_R \pths*{\bar{r} r}^k
	\end{equation}
for all $w\in\ball{\bar{x}}{\bar{r}}$ and all $0<r<1$, then 
	\begin{equation}
		\sum_{x\in\mC}r_x^k\leq C_R \bar{r}^k.
	\end{equation}
\end{corollary}

\begin{proof}
It suffices to apply \cref{theorem: Reifenberg} to the measure $\hat{\mu}_{\bar{x},\bar{r}}\doteq\bar{r}^{-k}T_{\bar{x},\bar{r}}\sharp\mu$ introduced in the remark above: by using the change of variable formula for the integral, and exploiting the scale invariance of $\beta^k$ we obtain the result.
\end{proof}

\begin{remark}
Notice that the constants $\delta_R$ and $C_R$ for the rescaled version are the same as in \cref{theorem: Reifenberg}, and thus only depending on $m$. 
\end{remark}

\subsection{Estimates on Jones' numbers}
The key estimate, linking the Jones' numbers of a measure with the normalized $p$-energy of a $p$-minimizing map, is given in the following theorem, which we prove in several steps.
\hiddennconst{C}{C: best approx}

\begin{theorem}[Estimates on $\beta^k_\mu$]\label{theorem: L2 estimates}
	Let $u\in\wonepN*{\Omega}$ be a $p$-energy minimizing map. Fix the following constants: $0<\bar{r}\leq 1,\eta>0,\sigma>1, k\in\brcs{1,\dots,m}$. Let $x\in\ball{0}{1}$ and $r>0$. Assume $u$ is \underline{not} $\pths*{\eta,k+1}$-invariant in $\ball{x}{\bar{r}r}$. There exists a constant $\oconst{C}{C: best approx}(m,p,\eta,\sigma, \bar{r})$ such that the following estimate
	\begin{equation} \tag{JN}
		\beta^k_\mu(x,r)^2\leq \oconst{C}{C: best approx} r^{-k}\int_{\ball{x}{r}}(\theta(y,\sigma r)-\theta(y,r))\,d\mu(y).
	\end{equation}
holds for any positive Radon measure $\mu$ on $\Omega$.
\end{theorem}

\newcounter{JN}%
\setcounter{JN}{1}

\begin{remark}
When this theorem will be used in \cref{section: covering}, we'll assign a precise value to $\sigma$; notice that in order that all the expressions involved are meaningful we could need to enlarge the domain $\Omega$ (and thus $\bar{R}$) according to $\sigma$ (see Assumptions on the domain at page \pageref{paragraph: assumptions on omega}).
\end{remark}

\begin{remark}
By scale invariance (of both $\beta^k_\mu$ and $\theta$), it will be enough to prove the estimate for $x=0$ and $r=1$. Moreover, since the inequality does not change when multiplied by a constant, we can assume $\mu$ is a probability measure on $\ball{0}{1}$. That is: assuming $u$ is not $\pths*{\eta,k+1}$-invariant in $\ball{0}{\bar{r}}$, we will prove that
	\begin{equation}
		\tag{JN\alph{JN}}
		\beta^k_\mu(0,1)^2\leq \oconst{C}{C: best approx} \int_{\ball{0}{1}}(\theta(y,\sigma )-\theta(y,1))\,d\mu(y)
		\stepcounter{JN}
	\end{equation}
for any measure $\mu$ with $\mu\pths{\ball{0}{1}}=1$
\end{remark}

\paragraph{Summary of the proof.} {\parskip=0pt
First of all, we find an explicit expression for the Jones' number of a measure. This relies on the fact that the infimum among affine planes which defines $\beta^k_\mu$ is actually achieved, and the minimal plane has a particularly manageable characterization (\cref{subtheorem L2 estimate 1}). \par
Secondly (and separately from the first point), we use directly the monotonicity formula \eqref{equation: monotonicity formula} to estimate the \inquotes{radial $p$-energy} of $u$ in a ball with the quantity $\theta(\cdot,\sigma)-\theta(\cdot,1)$ computed at the center of the ball (\cref{subtheorem: L2estimate 2}). \par
Next, we estimate the $p$-energy of $u$ along some selected directions $v_j\in\sphere{m}$, exploiting the information we got from the second step; here the explicit expression for $\beta^k_\mu$ emerges (\cref{subtheorem: L2estimate 3}). \par 
Finally, we consider $k+1$ selected directions together: we give an upper bound to the $p$-energy of $u$ along a $(k+1)$-plane in terms of $\theta(\cdot,\sigma)-\theta(\cdot,1)$. We then use the fact that $u$ is not $(k+1)$-almost invariant at $0$ to say that the same $p$-energy must be at least some fixed amount.}
\par
We begin by writing the Jones' number in an explicit way. Some preliminary definitions are needed.

\begin{definition}\label{def:cmassa}
Let $\mu$ be a measure with support in $\ball{0}{1}$. We define:
	\begin{enumerate}[(i)]
		\item the \textbf{center of mass} of $\mu$ as the point $\cm^\mu\in\ball{0}{1}$ \st{} 
		\begin{equation}
			\cm^\mu\doteq\dashint_{\ball{0}{1}}x\,d\mu(x);
		\end{equation}
		\item the \textbf{second moment} of $\mu$ as the bilinear form $Q^\mu$ \st{} for all $v,w\in\R^m$
		\begin{equation}
			Q^\mu(v,w)\doteq\int_{\ball{0}{1}}\bkts*{(x-\cm)\cdot v}\bkts*{(x-\cm)\cdot w}\,d\mu(x).
		\end{equation}
	\end{enumerate}
We'll usually drop the superscript $\mu$ when it is clear from the context.
\end{definition}
Since $Q$ is symmetric and positive-definite, by the Spectral Theorem the associated matrix (which we still denote by $Q$) admits an orthonormal basis of eigenvectors, with non-negative eigenvalues. We denote with $\lambda_1^\mu,\dots,\lambda_m^\mu$ the eigenvalues of $Q$ in decreasing order, and with $v_1^\mu,\dots,v_m^\mu$ the respective eigenvectors (pairwise orthogonal and of norm $1$), again dropping the superscripts when they are clear; in particular:
\begin{gather}
\lambda_kv_k=\int_{\ball{0}{1}}\bkts*{(x-\cm)\cdot v_k}(x-\cm)\,d\mu(x);\\
\lambda_1\geq\lambda_2\geq\dots\geq\lambda_m.
\end{gather}
Finally, we denote by $V^\mu_k$ (or $V^k$) the following affine $k$-plane:
	\begin{align}
		V^\mu_k&\doteq \cm^\mu+W^\mu_k\label{equation: vikappa}\\
		W^\mu_k&\doteq \spann\brcs*{v_1^\mu,\dots,v_k^\mu}.
	\end{align}
We are now ready to characterize $\beta^k_\mu$:
\newcounter{Ltwoestimates}
\setcounter{Ltwoestimates}{1}
\renewcommand{\thesubtheorem}{\ref*{theorem: L2 estimates}.\arabic{Ltwoestimates}}
\begin{subtheorem}\label{subtheorem L2 estimate 1}\stepcounter{Ltwoestimates}
Let $\mu$ be a measure on $\ball{0}{1}$. The affine space $V^\mu_k$ achieves the minimum in the definition of $\beta_\mu^k(0,1)$. Moreover,
	\begin{equation}
		\beta_\mu^k(0,1)=\int_{\ball{0}{1}}\dist^2\pths*{y,V^\mu_k}\,d\mu(y)=\lambda_{k+1}^\mu+\dots+\lambda_m^\mu.
	\end{equation}
\end{subtheorem}

For a proof of this fact, see \cite[Lemma 7.4]{Naber2017} or \cite[Subsection 6.1]{Naber2016}. This is based on the (visually helpful) fact that the eigenvalues $\lambda_k$ and eigenvectors $v_k$ admit the following characterization:
\begin{itemize}
	\item $\lambda_1$ is the maximum of $\int_{\ball{0}{1}}\ango*{x-\cm,v}^2\,d\mu$ among vectors $v$ of norm $1$, and $v_1$ is any maximizing vector;
	\item $\lambda_k$ is the minimum of the same operator among all unit vectors orthogonal to $v_1,\dots,v_{k-1}$, and $v_k$ is any maximizing vector.
\end{itemize}

We have thus reduced the problem to showing:
	\begin{equation}
		\tag{JN\alph{JN}}
		\lambda_{k+1}^\mu+\dots+\lambda_m^\mu \leq C \int_{\ball{0}{1}}(\theta(y,\sigma )-\theta(y,1))\,d\mu(y);
		\stepcounter{JN}
	\end{equation}
or, since the eigenvalues are ordered decreasingly, we need to show even less:
	\begin{equation}\label{equation: Jones number 4}
		\tag{JN\alph{JN}}
		\lambda_{k+1}^\mu \leq C \int_{\ball{0}{1}}(\theta(y,\sigma )-\theta(y,1))\,d\mu(y).
		\stepcounter{JN}
	\end{equation}
	
The following result gives an estimate involving the difference $\theta(y,\sigma)-\theta(y,1)$. It is really a direct consequence of the monotonicity formula (and our choice of $\psi$): none of the tools just introduced is required. Recall that $\xi$ and $t_a$ are the constants introduced in \cref{definition: psi}.

\hiddennconst{C}{C: subth2:1}
\hiddennconst{C}{C: subth2:2}

\begin{subtheorem}
\label{subtheorem: L2estimate 2}
\stepcounter{Ltwoestimates}
Let $u\in\wonepN*{\Omega}$ be $p$-minimizing, $y\in\ball{0}{1}$; assume $R_1, R_2$ and $R$ are radii satisfying $R_2\leq\frac{R}{t_a}<R_1$. Then the following inequality holds for some constant $\oconst{C}{C: subth2:1}(m,R,R_1,p)$ (provided both sides are well defined):
	\begin{equation}\label{equation: estimate on W}
		\int_{\ball{y}{R}}\abs*{\nabla u(z)}^{p-2}\abs*{\ango*{\nabla u(z),y-z}}^2\,dz\leq \oconst{C}{C: subth2:1}\pths*{\theta(y,R_1)-\theta(y,R_2)}.
	\end{equation}
In particular, taking $R_2=1$, $R= \max\brcs{2,t_a}$, and $R_1>R/t_a$ we also get
	\begin{equation}\label{equation: estimate on W bis}
		\int_{\ball{0}{1}}\abs*{\nabla u(z)}^{p-2}\abs*{\ango*{\nabla u(z),y-z}}^2\,dz\leq \oconst{C}{C: subth2:2}\pths*{\theta(y,R_1)-\theta(y,1)},
	\end{equation}
	with $\oconst{C}{C: subth2:2}$ depending only on $R_1, m, p$.
\end{subtheorem}
	
\begin{proof}
By \cref{theorem: monotonicity formula}:
	\begin{equation}
		\theta(y,R_1)-\theta(y,R_2)\geq\int_{\frac{R}{t_a}}^{R_1} pr^{p-m-2}\int_{\ball{y}{R}}\!\abs{z-y} \bkts*{-\psi'\pths*{\frac{\abs{z-y}}{r}}}\abs{\nabla u }^{p-2}\abs{\partial_{r_y}u}^2\,dz\,dr.
	\end{equation}
Since $z\in\ball{y}{R}$ and $r\geq \frac{R}{t_a}$, by the assumptions on $\psi$ made in \cref{definition: psi} we have:
	\begin{equation}
		-\psi'\pths*{\frac{\abs{z-y}}{r}}\geq \xi.
	\end{equation}
The two integrals can be then separated and we get:
	\begin{equation}
		\theta(y,R_1)-\theta(y,R_2)\geq p\xi\frac{R_1^{p-m-1}- \pths*{\frac{R}{t_a}}^{p-m-1}}{p-m-1}\int_{\ball{y}{R}}\abs{\nabla u(z) }^{p-2}\frac{\abs{\ango{\nabla u(z), z-y}}^2}{\abs{z-y}}\,dz.
	\end{equation}
But now notice that $\abs{z-y}^{-1}\geq R^{-1}$: inequality \eqref{equation: estimate on W} is proved. The last statement (\cref{equation: estimate on W bis}) follows from the fact that $\ball{0}{1}\subset\ball{y}{R}$, because $y\in\ball{0}{1}$ and $R\geq 2$.
\end{proof}	

Next, we use the \inquotes{radial} information we just achieved to estimate the $p$-energy along the eigenvectors $v_j$. Notice that the eigenvalues $\lambda_j$ also appear in the estimate.
\hiddennconst{C}{C:subth3:1}
\begin{subtheorem}
\label{subtheorem: L2estimate 3}
\stepcounter{Ltwoestimates}
Let $u$ be $p$-minimizing, and $\mu$ a Radon measure on $\ball{0}{1}$; let $\brcs{\lambda_j}_j$, $\brcs{v_j}_j$ be the eigenvalues and eigenvectors of $Q^\mu$, as before. There exists a constant $\oconst{C}{C:subth3:1}(m,\varn,p)$ such that for all $j=1,\dots ,m$ the following holds:
	\begin{equation}
	\lambda_j\int_{\ball{0}{1}}\abs{\nabla u}^{p-2}\abs*{\ango{\nabla u,v_j}}^2\,dz\leq \oconst{C}{C:subth3:1}\int_{\ball{0}{1}} \pths*{\theta(y,\sigma )-\theta(y,1)} \,d\mu(y).
	\end{equation}
\end{subtheorem}

\begin{proof}
Up to performing a translation, we can assume $\cm=0$.
By definition of the eigenvalue $\lambda_j$, we have
\begin{equation}
	\lambda_j v_j=\int_{\ball{0}{1}}\ango{y,v_j}y\,d\mu(y).
\end{equation}
Fix $z\in\ball{0}{1}$. By taking the scalar product of $\nabla u(z)$ with both sides of the previous equality, we obtain:
	\begin{equation}\label{equation: proof of l2estimate, 1}
	\lambda_j\ango{\nabla u(z),v_j}=\int_{\ball{0}{1}}\ango{y,v_j}\ango*{\nabla u(z),y}\,d\mu(y);
	\end{equation}
moreover, since $\cm=0$, we have:
	\begin{equation}\label{equation: proof of l2estimate, 2}
		\int_{\ball{0}{1}} \ango{y,v_j}\ango*{\nabla u (z), z}\,d\mu(y) = \ango*{\nabla u (z), z}\ango*{\int_{\ball{0}{1}} y\,d\mu(y),v_j}=0.
	\end{equation}
By subtracting \eqref{equation: proof of l2estimate, 2} to \eqref{equation: proof of l2estimate, 1}, and multiplying both sides by $\abs{\nabla u(z)}^{\frac{p}{2}-1}$, we can write:
	\begin{equation}
	\lambda_j\abs{\nabla u(z)}^{\frac{p}{2}-1}\ango*{\nabla u(z),v_j}=\int_{\ball{0}{1}}\ango{y, v_j}\abs{\nabla u(z)}^{\frac{p}{2}-1}\ango*{\nabla u(z),y-z}\,d\mu(y).
	\end{equation}
We then take the squared norms of both sides and apply H\"older inequality:
\begin{equation}\label{equation: proof of l2estimate 3}
\begin{split}
	\lambda_j^2\abs{\nabla u(z)}^{p-2}\abs*{\ango{\nabla u(z),v_j}}^2&\leq%
	\begin{multlined}[t][.5\textwidth]
		{\int_{\ball{0}{1}}\ango{y, v_j}^2 \,d\mu(y)}\cdot\\
		\cdot{\int_{\ball{0}{1}} \abs{\nabla u(z)}^{p-2}\abs*{\ango{\nabla u(z),y-z}}^2\,d\mu(y)}=
	\end{multlined} \\
	&=\lambda_j\int_{\ball{0}{1}}\abs{\nabla u(z)}^{p-2} \abs*{\ango{\nabla u(z),y-z}}^2\,d\mu(y).
\end{split}
\end{equation}
Now if $\lambda_j=0$ the statement is trivial, so we can assume $\lambda_j>0$ (since all the eigenvalues are non-negative); we are thus allowed to divide by $\lambda_j$. \Cref{equation: proof of l2estimate 3} holds for all $z\in\ball{0}{1}$, thus we can then integrate both sides on $\ball{0}{1}$ with respect to the Lebesgue measure in the variable $z$. At this point we get, also using Tonelli's Theorem:
\begin{equation}
	\lambda_j \int_{\ball{0}{1}}\abs{\nabla u(z)}^{p-2}\ango{\nabla u(z),v_j}^2\,dz\leq \int_{\ball{0}{1}} \int_{\ball{0}{1}} \abs{\nabla u(z)}^{p-2} \abs*{\ango{\nabla u(z),y-z}}^2\,dz\,d\mu(y).
\end{equation}
Hence, a direct application of \cref{subtheorem: L2estimate 2} gives the desired result.
\end{proof}

Now, thanks to the last result, we have an upper bound on the $p$-energy along the $(k+1)$-plane $V_k$ introduced in \cref{equation: vikappa}; but this is bounded from below by a constant, by the lack of almost invariance in $0$. We have all the ingredients to complete the proof of \cref{theorem: L2 estimates}.
 
\begin{proof}[Proof of \cref{theorem: L2 estimates}]
Applying \cref{subtheorem: L2estimate 3} to $\lambda_1,\dots,\lambda_{k+1}$, and recalling that the $\lambda_j$'s are ordered decreasingly, we get:
	\begin{equation}\label{eq: l2proof:1}
		\begin{split}
			\lambda_{k+1}\int_{\ball{0}{1}}\abs{\nabla u}^{p-2}\abs*{\ango*{\nabla u, V^\mu_{k+1}}}^2\,dz&\leq \sum_{j=1}^{k+1}\lambda_j\int_{\ball{0}{1}}\abs{\nabla u}^{p-2}\abs*{\ango{\nabla u,v_j}}^2\,dz\leq\\
			&\leq (k+1)\oconst{C}{C:subth3:1}\int_{\ball{0}{1}} \pths*{\theta(y,\sigma )-\theta(y,1)} \,d\mu(y).
		\end{split}
	\end{equation}
On the other hand: since $u$ is not $(\eta,k+1)$-invariant in $\ball{0}{\bar{r}}$, and $V^\mu_{k+1}$ is a $(k+1)$-plane, we have by definition:
	\begin{equation}\label{eq: l2proof:2}
	\begin{split}
		\int_{\ball{0}{1}}\abs{\nabla u}^{p-2}\abs*{\ango*{\nabla u, V^\mu_{k+1}}}^2\,dz&\geq \int_{\ball{0}{1}}\abs*{\ango*{\nabla u, V^\mu_{k+1}}}^p\,dz\geq\\
		&\geq \int_{\ball{0}{\bar{r}}}\abs*{\ango*{\nabla u, V^\mu_{k+1}}}^p\,dz\geq \bar{r}^{m-p}\eta.
	\end{split}
	\end{equation}
In particular, putting together \cref{eq: l2proof:1} and \cref{eq: l2proof:2} we obtain:
	\begin{equation}
		\lambda_{k+1}\leq \frac{(k+1)\oconst{C}{C:subth3:1}(\sigma,m,p)}{\eta \bar{r}^{m-p}}\int_{\ball{0}{1}} \pths*{\theta(y,\sigma )-\theta(y,1)} \,d\mu(y),
	\end{equation}
which is \eqref{equation: Jones number 4}.
\end{proof}

\section{A collection of structural lemmas}\label{section:awful}
This section is devoted to building a series of \inquotes{quantitative} geometric results about $p$-minimizing mappings, describing the behavior of some special subsets of $\singu$. Analogous results for (approximate) $2$-harmonic maps can be found in \cite[Section 4]{Naber2016}, although stated with some differences.
\par 
In \cref{lemma: tech lemma 1,lemma: tech lemma 2,lemma: tech lemma 3}, recall that $t_a, t_b$ are the structural constants introduced in \cref{definition: psi}, describing some particular features of $\psi$. Moreover, in each of those results we implicitly assume that $u\in\wonepN{\Omega}$ is a $p$-energy minimizing map with energy bounded by $\Lambda$, and the constants we find are independent of $u$.
\par
In the first lemma we convey this idea: consider the set of points in $\ball{0}{1}$ at which $\theta[u]$ satisfies a suitable pinching condition; if it spans $\rho$-effectively a $k$-dimensional plane $L$, then for some $\delta>0$ the stratum $\strat{k}{\eta}{\delta}$ lies inside a fattening of $L$. See \cref{definition: effective span} for the definition of a set effectively spanning a $k$-plane.
\begin{lemma}\label{lemma: tech lemma 1}
	Let $\rho_1,\lambda_1,\eta>0$ and $0<\gamma<1$ be (small enough) constants. Define $c(\gamma)\doteq \half (1-\gamma)t_a$, where $t_a$ is introduced in \cref{definition: psi}. There exist constants $\delta_0,\epsilon$ (depending on $m, p, \varn, \Lambda$ and on the parameters just introduced) such that the following holds: if the set 
	\begin{equation}
	\mathcal{K}\doteq \brcs*{y\in\ball{x}{c r}\stset \theta(y,r)-\theta(y,\lambda_1 r)<\epsilon}
	\end{equation}
	spans $\rho_1 r$-effectively a $k$-plane $L$, then $\strat{k}{\eta}{\delta_0 r}\cap\ball{x}{c r}\subset  \fat{L}{\rho_1 r}$.
\end{lemma}

\textit{In the proof we drop the subscript $1$ on $\rho$ and $\lambda$; it was introduced so that lemma is easier to recall when we need it.}

\begin{proof}[Sketch] Assume \wolog{} that $x=0$, $r=1$. We are thus assuming that 
	\begin{equation}
	\mK=\brcs*{y\in\ball{0}{c}\stset\theta(y,1)-\theta(y,\lambda)<\epsilon} \quad \text{spans $\rho$-effectively $L\in\aff{k}$},
	\end{equation}
with $\lambda$ and $\rho$ fixed and $\epsilon$ to be chosen.
For a fixed point $w$ out of $\fat{L}{\rho}$ we need to show that 
	\begin{equation}
		\tau^{p-m}\int_{\ball{w}{\tau}}\abs*{\nabla_V u}^p <\eta
	\end{equation}
	for some $\tau\geq \delta_0$, with $\delta_0$ depending only on $m, p,\varn,\Lambda, \eta,\rho,\lambda,\gamma$, and for some $k+1$-dimensional plane $V$. Denote by $y_0,\dots,y_k$ a set of points of $\mathcal{K}$ which $\rho$-effectively span $L$.
\begin{stepizec}{Step}
	\item By \cref{subtheorem: L2estimate 2}, with $y\in\mK$, $R_1=1$, $R_2=\lambda$ and $\max\brcs{\lambda , 1-\gamma}t_a < R <t_a$, we have 
	\hiddennconst{c}{c: tch lemma 1: 1}
		\begin{equation}
			\int_{\ball{y}{R}}\abs*{\nabla u(z)}^{p-2}\abs*{\ango*{\nabla u(z),y-z}}^2\,dz\leq \oconst{c}{c: tch lemma 1: 1}\pths*{\theta(y,1)-\theta(y,\lambda_1)}\leq \oconst{c}{c: tch lemma 1: 1}\epsilon.
		\end{equation}
	\item Notice that $\ball{0}{c}\subset\ball{y}{R}$ for any $y\in\mK$ and $R$ as in the previous point; then for any such $y$: 
		\begin{equation}
			\int_{\ball{0}{c}}\abs*{\nabla u}^{p-2}\abs*{\ango{\nabla u, y-z}}^2\,dz\leq \oconst{c}{c: tch lemma 1: 1}\epsilon.
		\end{equation}
		Define $\hat{L}$ to be the linear subspace associated to $L$. For any $v\in\hat{L}$ of norm $1$, we have that 
		\hiddennconst{c}{c: tch lemma 1: 2}
		\begin{equation}
			v=\sum_{i=1}^k\alpha_i(y_i-y_0), \qquad \abs{\alpha_i}\leq \oconst{c}{c: tch lemma 1: 2}(m).
		\end{equation}
		Hence, by a standard estimate, 
		\hiddennconst{c}{c: tch lemma 1: 3}
		\begin{equation}
		\begin{split}
			\int_{\ball{w}{\tau}}\abs*{\nabla u(z)}^{p-2}\abs{\nabla_{v}u(z)}^2\,dz&\leq 2\sum_i\alpha_i^2\int_{\ball{0}{c}} \abs*{\nabla u}^{p-2}\abs*{\ango{\nabla u, y_i-z}}^2\,dz+\\
			&+2\pths*{\sum_i\alpha_i}^2\int_{\ball{0}{c}} \abs*{\nabla u}^{p-2}\abs*{\ango{\nabla u,z-y_0}}^2\,dz\leq\\
			&\leq\oconst{c}{c: tch lemma 1: 3}(m)\epsilon.
		\end{split}
		\end{equation}				
		As a consequence, if $\brcs*{v_1,\dots,v_k}$ is an orthonormal basis of $\hat{L}$,
		\hiddennconst{c}{c: tch lemma 1: 4}
		\begin{equation}
		\int_{\ball{w}{\tau}}\abs*{\nabla u(z)}^{p-2}\sum_{i=1}^k\abs*{\nabla_{v_i}u(z)}^2\,dz\leq \oconst{c}{c: tch lemma 1: 4}(m)\epsilon;
		\end{equation}
		thus, along $k$ directions, we have some information that goes in the direction we need (recall that $\epsilon$ is a constant we still have to choose).
	\item For points lying out of $\fat{L}{\rho}$ we need to gain another direction of smallness, and we do it by considering the orthogonal direction to $L$ passing through $w$: we can estimate the quantity
		\begin{equation}
			\int_{\ball{w}{\tau}}\abs*{\nabla u}^{p-2}\abs*{\ango*{\nabla u(z),\frac{z-\pi_L(z)}{\abs{z-\pi_L(z)}}}}^2
		\end{equation}
		using the same technique of the first point, and then exploit the triangle inequality to achieve an estimate on 
		\begin{equation}\label{eq: tech lemma 1}
			\int_{\ball{w}{\tau}}\abs*{\nabla u}^{p-2}\abs*{\ango*{\nabla u(z),\frac{w-\pi_L(w)}{\abs{w-\pi_L(w)}}}}^2.
		\end{equation}
More precisely: assume now that	$\tau<\frac{\rho}{2}$; define for the sake of simplicity $h(z)\doteq \frac{z-\pi_L(z)}{\abs{z-\pi_L(z)}}$, where $\pi_L$ is the orthogonal projection onto $L$. Similarly to what happened in the previous step, we have 
		\hiddennconst{c}{c: tch lemma 1: 5}
		\hiddennconst{c}{c: tch lemma 1: 6}
		\begin{gather}
			\pi_L(z)=y_0+\sum_{i=1}^k\beta_i(z)(y_i-y_0), \qquad \abs{\beta_i(z)}\leq \oconst{c}{c: tch lemma 1: 5}(m)\\
			\Longrightarrow\quad \int_{\ball{w}{\tau}}\abs{\nabla u}^{p-2}\abs{\ango{\nabla u, z-\pi_L(z)}}^2\,dz\leq \oconst{c}{c: tch lemma 1: 6}(m)\epsilon,
		\end{gather}
so that 
		\begin{equation}\label{eq: tech lemma 1: 2}
		 \int_{\ball{w}{\tau}}\abs*{\nabla u}^{p-2}\abs*{\ango*{\nabla u, \frac{z-\pi_L(z)}{\abs*{z-\pi_L(z)}}}}^2\,dz\leq 4\oconst{c}{c: tch lemma 1: 6}(m)\epsilon\rho^2.
		\end{equation}
Moreover, for any $z\in\ball{w}{\tau}$ we have that $z=w+\tau v^\perp+\tau v^\top$, with $v^\top\in\hat{L}\cap \ball{0}{1}$, $v^\perp\in L^\perp\cap\ball{0}{1}$. In that case, the projection of $z$ onto $L$ can be written as $\pi_L(z)=\pi_L(w)+\tau v^\top$. As a consequence, we can estimate
		\begin{equation}\label{eq: tech lemma 1: 3}
		\begin{split}
		\abs*{h(z)-h(w)}&=\abs*{\frac{w-\pi_L(w)+\tau v^\perp}{\abs*{z-\pi_L(z)}}-\frac{w-\pi_L(w)}{\abs*{w-\pi_L(w)}}}=\\
		&=\abs*{\pths*{\frac{1}{\abs*{z-\pi_L(z)}}-\frac{1}{\abs*{w-\pi_L(w)}}}(w-\pi_L(w))+\frac{\tau v^\perp}{\abs*{z-\pi_L(z)}}}\leq \\
		&\leq \frac{\abs*{\abs*{w-\pi_L(w)}-\abs*{z-\pi_L(z)}}+\tau}{\abs*{z-\pi_L(z)}}\leq \frac{4\tau}{\rho}.
		\end{split}
		\end{equation}
To estimate the expression in \eqref{eq: tech lemma 1}, we exploit \cref{eq: tech lemma 1: 2,eq: tech lemma 1: 3}, the bound on $\theta(w,\tau)$ and a Cauchy-Schwarz inequality:
		\begin{equation}
		\begin{split}
		\int_{\ball{w}{\tau}}\abs*{\nabla u}^{p-2}\abs*{\ango*{\nabla u(z),h(w)}}^2\,dz&\leq
		\begin{multlined}[t]		
		 2\int_{\ball{w}{\tau}}\abs*{\nabla u}^{p-2}\abs*{\ango*{\nabla u, h(z)}}^2\,dz+\\
		 +2\int_{\ball{w}{\tau}}\abs*{\nabla u}^{p-2}\abs*{\ango*{\nabla u, h(w)- h(z)}}^2\,dz\leq
		\end{multlined}\\
		&\leq 4\oconst{c}{c: tch lemma 1: 6}(m)\epsilon\rho^2+\frac{32\Lambda}{\rho^2}\tau^{2-(p-m)}.
		\end{split}
		\end{equation}
\item Putting together the previous steps, we consider $V=\hat{L}\oplus h(w)$: a simple computation gives:
		\begin{equation}
		\begin{split}
		\tau^{p-m}\int_{\ball{\tau}{w}}\abs*{\nabla_V u}^p\,dz&\leq
		\begin{multlined}[t]
		\tau^{p-m}\int_{\ball{w}{\tau}}\abs*{\nabla u(z)}^{p-2}\sum_{i=1}^k\abs*{\nabla_{v_i}u(z)}^2\,dz+\\
		+\tau^{p-m}\int_{\ball{w}{\tau}}\abs*{\nabla u}^{p-2}\abs*{\ango*{\nabla u(z),h(w)}}^2\,dz\leq
		\end{multlined}
		\\
		&\leq \oconst{c}{c: tch lemma 1: 4}(m)\epsilon\tau^{p-m}+4\oconst{c}{c: tch lemma 1: 6}(m)\rho^2\epsilon\tau^{p-m}+\frac{32\Lambda}{\rho^2}\tau^{2}.
		\end{split}
		\end{equation}	
Thus, in order to conclude, we only need to choose $\tau$ so that the last term is smaller than $\frac{\eta}{2}$, and then choose $\epsilon$ such that also the sum of the first two pieces is smaller than $\frac{\eta}{2}$.
		\qedhere
\end{stepizec}
\end{proof}

The upcoming lemma says the following: if we have a set of points that
satisfy a suitable pinching condition on $\theta$, and they effectively span a $k$-
subspace $L$, then all the points of $L$ inherit a (possibly weaker) pinching
condition. This can be seen as a further quantitative version of \cref{corollary:  rigidity} (it is indeed applied to the limit of a contradicting sequence).

\begin{lemma}\label{lemma: tech lemma 2}
Let $\rho_2$, $\lambda_2$, $\Lambda$, $\gamma>0$. Let $0<c<\half t_b$. There exists a constant $\delta(\rho_2,\lambda_2, \Lambda, \gamma, c)$ \st{} the following holds: if $\theta(y,r)\leq E$ for all $y\in\ball{x}{cr}\cap \mS$ (with $\mS\subset\Omega$ and $E\leq \Lambda$), and the set
	\begin{equation}
		\mH \doteq \brcs*{y\in\ball{x}{cr}\cap \mS\stset \theta(x,\lambda_2 r)>E-\delta}
	\end{equation}
spans $\rho_2 r$-effectively a $k$-space $L$, then we have
	\begin{equation}
		\theta (z,\lambda_2 r)>E-\gamma
	\end{equation}
for all $z\in\ball{x}{cr}\cap L$.
\end{lemma}

\textit{In the proof we drop the subscript $2$ on $\rho$ and $\lambda$; it was introduced so that lemma is easier to recall when we need it.}

\begin{proof}
Assume $x=0$, $r=1$. If the statement is false, one can find a sequence of $p$-minimizing maps $\brcs{u_i}_i$, $k+1$ sequences of points $\brcs{y_{ij}}_{i\in\N}$ in $\ball{0}{c}\cap\mS$ (with $j=1,\dots,k+1$), and a further sequence $\brcs{z_i}_i$ in $\ball{0}{c}$ such that:
	\begin{itemize}
	\item $\brcs{y_{ij}}_{j=1}^{k+1}$ spans $\rho$-effectively a $k$-space $L$ (which can be assumed to be the same for all $i$).
	\item $\theta[u_i](y_{ij},\lambda)>E-\frac{1}{i}$, and $\theta[u_i](y_{ij},1)\leq E$.
	\item $\theta[u_i](z_i,\lambda)<E-\gamma$, and $z_i\in L$.
	\end{itemize}
Up to subsequences, $\brcs{u_i}$ converges in $W^{1,p}$ to a $p$-minimizing map $\bar{u}$, $y_{ij}\to \bar{y}_j\in\clball{0}{c}$, and $z_i\to \bar{z}\in\clball{0}{c}\cap L$. Moreover, the set of points $\brcs{\bar{y}_j}_{j=1}^{k+1}$ still spans $L$, and we have $\theta[\bar{u}](\bar{y}_j,\lambda)=E$, which implies 
	\begin{equation}
		\theta[\bar{u}](\bar{y}_j,1)-\theta[\bar{u}](\bar{y}_j,\lambda)=0.
	\end{equation}
Thus $\bar{u}$ is $L$-invariant in $\clball{0}{c}$ by \cref{corollary: rigidity}; so in particular both $\theta[\bar{u}](\bar{z},\lambda)=E$ and $\theta[\bar{u}](\bar{z},\lambda)\leq E-\gamma$ should hold.
\end{proof}

Finally, a result which states that \emph{the lack of almost invariance spreads uniformly along pinched points}. This is yet another quantitative rephrasing of the fact that if $\theta(\cdot,r)-\theta(\cdot,\lambda r)=0$ at two different (close) points, then $u$ is invariant along the direction connecting them.
\begin{lemma}\label{lemma: tech lemma 3}
Let $\lambda>0$, $\sigma_0\in\pths*{0,\half t_b}$, $\kappa_0\in\pths*{0,1}$. There exists a constant $\epsilon$ such that the following holds. If the following conditions are satisfied by a pair of points $x,y$:
	\begin{enumerate}
		\item $\abs*{x-y}<\half t_b r$;
		\item $\theta\pths{x,r}-\theta\pths{x,\lambda r}<\epsilon$;
		\item $\theta\pths{y,r}-\theta\pths{y,\lambda r}<\epsilon$;
	\end{enumerate}
and $u$ is not $(\eta,k)$-invariant in $\ball{x}{\sigma r}$ for some $\sigma_0\leq\sigma\leq \half t_b$, then $u$ is not $(\kappa_0\eta,k)$-invariant in $\ball{y}{\sigma r}$.
\end{lemma}

\begin{proof}
Assume $x=0$, $r=1$. By contradiction, there exist: a sequence $\brcs*{u_i}_i$ of $p$-minimizing maps, a sequence $\brcs*{y_i}_i$ of points in $\ball{0}{\half t_b}$ and a sequence $\brcs*{\sigma_i}_i$ in $\bkts{\sigma_0,\half t_b}$ such that:
	\begin{alignat}{2}
	\theta\bkts*{u_i}\pths{0,1}-\theta\bkts*{u_i}\pths{0,\lambda}&<\frac{1}{i}&\label{equation: techlemma 4 eq 1}\\
	\theta\bkts*{u_i}\pths{y_i,1}-\theta\bkts*{u_i}\pths{y_i,\lambda}&<\frac{1}{i}&\label{equation: techlemma 4 eq 2}\\
	\sigma_i^{p-m}\int_{\ball{0}{\sigma_i}}\abs*{\nabla_{L}u_i}^p&>\eta\qquad&\forall L\in\grass{k}\\
	\sigma_i^{p-m}\int_{\ball{y_i}{\sigma_i}}\abs*{\nabla_{\tilde{L}}u_i}^p& < \kappa_0 \eta \qquad&\exists \tilde{L}\in\grass{k}
	\end{alignat}
Up to subsequences, they converge, respectively, to a $p$-minimizing map $\bar{u}$ (in $W^{1,p}$), to a point $y\in\clball{0}{\half t_b}$ and to a number $\bar{\sigma}\in\bkts*{\sigma_0,\half t_b}$. Moreover, due to \cref{equation: techlemma 4 eq 1,equation: techlemma 4 eq 2} and by the $L^p$-convergence of gradients, $\bar{u}$ is $0$-homogeneous \wrt{} both $x$ and $y$, thus it is invariant along the direction $x-y$ in $\ball{0}{\half t_b}$. Again by the fact that $u_i\to\bar{u}$ in $\wonepN*{\Omega}$, however, we also have that
	\begin{align}
		\bar{\sigma}^{p-m}\int_{\ball{0}{\sigma}}\abs*{\nabla_{\tilde{L}}\bar{u}}^p&\geq \eta \\
		\bar{\sigma}^{p-m}\int_{\ball{\bar{y}}{\sigma}}\abs*{\nabla_{\tilde{L}}\bar{u}}^p&\leq \kappa_0\eta;
	\end{align}
	this contradicts the fact that the two left hand sides should be equal (by translation invariance of $\bar{u}$).
\end{proof}

\begin{remark}[Assumptions on $\psi$, again]\label{remark: assumptions on psi}
The structural constants $t_a$ and $t_b$ we introduced in \cref{definition: psi} were broadly used in these last few lemmas. In practice, the key feature (of $\psi$) we need for our purposes is the possibility to work handily in the ball $\ball{0}{1}$ (so for example we require that $1$ is an admissible value for $c(\gamma)$ in \cref{lemma: tech lemma 1}). In the end, with the choice $t_a>2$ done in \cref{definition: psi}, we can apply \cref{lemma: tech lemma 1,lemma: tech lemma 2} with $c=1$ and \cref{lemma: tech lemma 3} with $\abs*{x-y}<r$.
\end{remark}  

\section{Covering arguments}\label{section: covering}
In the sequel, if $B=\ball{x}{r}$ and $k$ is a constant, we denote by $kB$ the ball $\ball{x}{kr}$. We first give two useful definitions of \inquotes{sets of points satisfying a pinched condition}; we have (more or less) already used both of them in \cref{section:awful}.

\begin{definition}
Let $u$ be a $p$-minimizing map, $x\in\ball{0}{1}$, $r>0$. Assume $E,\eta,\lambda,\delta>0$ are fixed, and $\mS\subset\ball{0}{1}$. We define
	\begin{align*}
		\mH(x,r)&=\mH^{\mS}_{E,\delta,\lambda}(x,r)\doteq\brcs*{y\in\ball{x}{r}\cap \mS\stset \theta\pths*{y,\lambda r}>E-\delta} \\
		\mK(x,r)&=\mK^{\mS}_{\delta,\lambda}(x,r)\doteq\brcs*{y\in\ball{x}{r}\cap \mS\stset \theta\pths*{y,r}-\theta\pths*{y,\lambda r}<\delta}.
	\end{align*}
If $B=\ball{x}{r}$, we also denote by $\mH_B, \mK_B$ the sets $\mH(x,r),\mK(x,r)$ respectively.
\end{definition}

It is clear that, if all the parameters appearing are fixed, and $u$ is such that $\theta\pths{y,r}\leq E$ for all $y\in\ball{x}{r}$, then $\mH(x,r)\subset\mK(x,r)$. As a consequence, whenever $\mH$ $\rho$-effectively spans a $k$-subspace, also $\mK$ trivially does. Heuristically $\mH$ should be thought as a set of pinched points at which \emph{$\theta(y,\lambda r)$ has a value which is close to the maximum possible}.
 
\paragraph{Notations and map of the constants.}\label{paragraph: map of constants} This will be the context for the whole section:
	\begin{itemize}[label={--}]
	\item $m,\varn,p,\Lambda$ are fixed as in the previous sections (respectively: dimension of the domain, target manifold, exponent for the energy, upper bound on the $p$-energy).
	\item $u\in\wonepN*{\Omega}$ is a $p$-minimizing harmonic map with $p$-energy bounded by $\Lambda$.
	\item We let $\rho>0$ be a fixed constant, and $r=\rho^{\hj}$ for some $\hj\in\N_{\geq 1}$. The radius $r$ will be the scale parameter for the singular stratification, up to a constant. \emph{The constant $\rho$ will be arbitrary in the first covering, and will be then suitably selected in the construction of the second covering.}
	\item $\eta>0$ is the (fixed) closeness parameter for the stratification.
	\item $k\in\brcs*{1,\dots,m}$ is the dimension parameter for the stratification.
	\item $\gamma>0$ is a constant used for the pinching condition on $\theta$. \emph{It will be arbitrary in the construction of the first covering, then selected in \cref{proposition: volume estimates}.}
	\item $\delta_0$ is the constant introduced in \cref{lemma: tech lemma 1} and $\delta$ is the constant produced by \cref{lemma: tech lemma 2}, both depending on $\rho,\eta,\gamma$.
	\item $\mS$ is a subset of the stratum $\strat{k}{\eta}{\delta_0 r}\cap \ball{0}{1}$.
	\item $0<E\leq \Lambda$ is such that $\theta (x,1)\leq E$ for all $x\in\ball{0}{1}\cap\mS$.
	\end{itemize}
The goal of the upcoming constructions will be to build a \inquotes{controlled} covering of $\mS$. In words, the ultimate goal will be to cover $\mS$ with balls $B$ satisfying the following:
	\begin{enumerate}
	\item The sum of the $k^{\text{th}}$ powers of the radii is bounded by a universal constant.
	\item Up to rescaling by a fixed constant, the balls are pairwise disjoint.
	\item Either the radius of $B$ is less or equal to the fixed radius $r$; or the (normalized) $p$-energy in $B$ is lower than the \inquotes{maximal initial $p$-energy} $E$ by a fixed amount $\delta$ ($E$ and $\delta$ were introduced in the previous list of constants). In the latter case, we say that $B$ satisfies a uniform energy drop condition (see the below \cref{definition: uniform energy drop}).
	\end{enumerate}		
	This will be achieved in \cref{prop: second covering}. Once we have this, we can then apply the same reasoning to each of the balls where the $p$-energy drops uniformly (while keeping the other balls as they are). At each step other balls of radius $\leq r$ are produced, while the $p$-energy continues to drop uniformly in all the other balls. The procedure lasts a finite number of steps, until there's no energy left: indeed the total initial $p$-energy was bounded by a fixed constant $\Lambda$. This is the content of \cref{subsection: volume estimate}. Let's give a precise definition of energy drop:
	\begin{definition}\label{definition: uniform energy drop}
We say that a ball $B$ satisfies the uniform $(\lambda,\delta)$-energy drop condition if 
	\begin{equation}
		\theta\pths*{y, \lambda r_B}\leq E-\delta\qquad\text{for all $y\in \mS\cap B$};
	\end{equation}
From now on $\lambda=\quint$ and $\delta$ will be fixed ($\delta$ as in the Map of the Constants above), so we omit them and simply say \inquotes{uniform energy drop}.
	\end{definition}
	At first, we are only able to reach a partial result: we don't manage to fully get a uniform energy drop condition on the balls of the covering; but we can show that, in each ball, the points for which the $p$-energy does not drop uniformly lie close to a $(k-1)$-plane (so in the end they can be controlled very efficiently). This is the content  of the next subsection.

\newcommand{\tx}[1]{\textrm{\normalfont #1}}
\newcommand{\fa}[1]{\mG_{#1}}
\newcommand{\fb}[1]{\mE_{#1}}

\subsection{First covering}\label{subsection: first covering}
Recall that $r=\rho^\hj$ is the scale parameter of the singular stratum we are considering. Here $\rho>0$ is a fixed parameter and $\hj\in\N$: we are allowed to work with constants which depend on $\rho$, but \emph{not} on $\hj$.
\paragraph{Construction of the first covering}
We construct a covering $\mF$ of $\mS$ with the following properties:
	\begin{enumerate}
	\item $\mF=\fa{0}\cup\dots\cup\fa{\hj-1}\cup\fb{\hj}$. If $B\in\fa{h}$, then $B=\ball{x}{\rho^{h}}$ for some $x$. If $B\in\fb{\hj}$, then $B=\ball{x}{\rho^{\hj}}$ for some $x$. Mnemonic rule: when the construction of $\mF$ is complete, the subcovering $\fb{\hj}$ is made of balls with radius \emph{equal} to $r$, while the subcoverings labeled with $\mG$ are made of balls with radius \emph{greater} than $r$.
	\item If $\ball{x}{\rho^{h}}\in\fa{h}$ with $0\leq h \leq \hj-1$,  then
		$\mH\pths{x,\rho^{h}}\subset\fat{V}{\quint \rho^{h+1}}$ for some $(k-1)$-affine subspace $V\in\aff{k-1}$; here $\mH=\mH^{\mS}_{E,\delta,\quint\rho}$.
	\item If $B,B'\in\mF$ and $B\neq B'$, then $\quint B\cap\quint B'=\emptyset$.
	\item\label{covering 1 property 4} If $\ball{x}{r_x}\in\fa{1}\cup\dots\cup\fa{\hj-1}\cup\fb{\hj}$, then
		\begin{gather}
		\theta\pths*{x,\quint r_x}>E-\gamma,\label{equation: cov1prop4part1}\\
		x\in\strat{k}{\half \eta}{\delta_0r_x} \label{equation: cov1prop4part2}.
		\end{gather}
	\end{enumerate}
The strategy will be to apply inductively the lemmas from \cref{section:awful} at different scales. We thus proceed inductively on $j\in\brcs*{0,\dots,\hj}$.	 
\paragraph{Step 1, case A.} If $\mH(0,1)$ is contained in $\fat{V}{\quint \rho}$ with $V\in\aff{k-1}$, then we define $\fa{0}\doteq \brcs*{\ball{0}{1}}$. The other subcoverings are left empty, and the process stops here.
\paragraph{Step 1, case B.} Otherwise, $\mH(0,1)$ spans $\quint\rho$-effectively a $k$-space $L\pths{0,1}\in\aff{k}$. Thus, by \cref{lemma: tech lemma 1} with $\lambda_1=\rho_1=\quint \rho$, $\mS\cap \ball{0}{1}$ is contained in $\fat{L(0,1)}{\quint \rho}$. By \cref{lemma: tech lemma 2}, with $\lambda_2=\rho_2=\quint \rho$, for any $z\in L(0,1)\cap \ball{0}{1}$ we have
	\begin{equation}
	\theta \pths*{z,\quint\rho}>E-\gamma.
	\end{equation}
If $\gamma$ is small enough (smaller than a constant depending on $m,p,\eta$), by \cref{lemma: tech lemma 3} with $\lambda = \quint\rho$ and $\sigma = \delta_0\rho$, for any $z\in L(0,1)\cap \ball{0}{1}$ we have $z\in\strat{k}{\half \eta}{\delta_0\rho}$ (because $\strat{k}{\eta}{\delta_0\rho}\supset \strat{k}{\eta}{\delta_0 r}$). Cover $\mS\cap\ball{0}{1}$ with balls of radius $\rho$ with centers in $L(0,1)$ and such that $\quint B\cap \quint B'=\emptyset$ if $B\neq B'$. Call $\fb{1}$ this covering.\par
If  $\hj=1$, \ie{} the final radius $r$ we want to reach is $\rho^1$, then we can stop here the procedure. Otherwise, assume that $B\doteq\ball{x}{\rho}\in\fb{1}$ is a ball produced by Step 1, case B. 
\paragraph{Step 2, case A.} If $\mH(x,\rho)\subset \fat{V}{\quint \rho^2}$ for some $V\in\aff{k-1}$, then $B$ is one of the balls that we want to keep in our final covering $\mF$; we define 
	\begin{equation}
	\fa{1}\doteq \brcs*{\ball{x}{\rho}\in\fb{1}\stset \mH(x,\rho)\subset \fat{V}{\quint \rho^2} \tx{ for some $V\in\aff{k-1}$}}.
	\end{equation}
\paragraph{Step 2, case B.} If instead $B\notin\fa{1}$, this means that $\mH(x,\rho)$ spans $\quint \rho^2$-effectively a $k$-space $L(x,\rho)$. Thus, applying \cref{lemma: tech lemma 1,lemma: tech lemma 2,lemma: tech lemma 3} with the \textit{same constants} as in Step 1, Case B, we get:
	\begin{enumerate}
	\item $\mS\cap \ball{x}{\rho}\subset \fat{L(x,\rho)}{\quint\rho^2}$ for some $L(x,\rho)\in\aff{k}$;
	\item $\theta(z,\quint\rho^2)>E-\gamma$ for all $z\in L(x,\rho)\cap\ball{x}{\rho}$;
	\item $z\in\strat[u]{k}{\half \eta}{\delta_0\rho^2}$ for all $z\in L(x,\rho)\cap\ball{x}{\rho}$.
	\end{enumerate}
Now we cover $\mS\cap\ball{x}{\rho}\setminus\bigcup\fa{2}$ with balls of radius $\rho^2$ such that for any pair $B\neq B'$ of such balls we have $\quint B\cap \quint B'=\emptyset$ and $\quint B\subset \ball{x}{\rho}\setminus\fa{2}$; define $\fb{2,x}$ such a covering. Define \begin{equation}
\fb{2}\doteq\bigcup\brcs*{\fb{2,x}\stset B(x,\rho)\in\fb{1}\setminus \fa{2}}.
\end{equation}
This concludes Step 2.
\par After the $j^{\tx{th}}$ step, we have:
	\begin{itemize}
	\item $j$ families of balls $\fa{0},\dots,\fa{j-1}$, with the following properties: if $B\in\fa{h}$ then $B=\ball{x}{\rho^{h}}$ for some $x$, and $\mH\pths*{x,\rho^{h}}$ is contained in $\fat{V}{\quint\rho^{h+1}}$ for some $V\in\aff{k-1}$;
	\item A family $\fb{j}$ of balls of radius $\rho^{j}$.
	\end{itemize}
If $j=\hj$, then we are done. Otherwise, we proceed in the same fashion. Let $B=\ball{x}{\rho^j}\in\fb{j}$.
\paragraph{Step $j+1$, case A.} If $\mH(x,\rho^j)\subset \fat{V}{\quint \rho^{j+1}}$ for some $V\in\aff{k-1}$, then $B$ is one of the balls that we want to keep in our final covering $\mF$; we define 
	\begin{equation}
	\fa{j}\doteq \brcs*{\ball{x}{\rho^j}\in\fb{j}\stset \mH(x,\rho^j)\subset \fat{V}{\quint \rho^{j+1}} \tx{ for some $V\in\aff{k-1}$}}.
	\end{equation}
\paragraph{Step $j+1$, case B.} If instead $B\notin\fa{j}$, this means that $\mH(x,\rho^j)$ spans $\quint \rho^{j+1}$-effectively a $k$-space $L(x,\rho^j)$. Thus, applying \cref{lemma: tech lemma 1,lemma: tech lemma 2,lemma: tech lemma 3} with the \textit{same constants} as in Case B of the previous steps, we get:
	\begin{enumerate}
	\item $\mS\cap B\subset \fat{L(x,\rho^j)}{\quint\rho^{j+1}}$ for some $L(x,\rho^j)\in\aff{k}$;
	\item $\theta(z,\quint\rho^{j+1})>E-\gamma$ for all $z\in L(x,\rho^j)\cap\ball{x}{\rho^j}$;
	\item $z\in\strat[u]{k}{\half \eta}{\delta_0\rho^{j+1}}$ for all $z\in L(x,\rho^j)\cap\ball{x}{\rho^j}$.
	\end{enumerate}
Now we cover $\mS\cap\ball{x}{\rho^j}\setminus\bigcup_{h\leq j}\bigcup\fa{h}$ with balls of radius $\rho^{j+1}$ such that for any pair $B\neq B'$ of such balls we have $\quint B\cap \quint B'=\emptyset$ and $\quint B\subset \ball{x}{\rho}\setminus\bigcup_{h\leq j}\bigcup\fa{h}$; define $\fb{j+1,x}$ such a covering. Define \begin{equation}
\fb{j+1}\doteq\bigcup\brcs*{\fb{j+1,x}\stset B(x,\rho)\in\fb{j}\setminus \fa{j}}.
\end{equation}
Iterating the procedure until $\hj$, we obtain the desired construction.

\begin{definition}\label{definition: centers}
If $\mF=\fa{0}\cup\dots\cup\fa{\hj-1}\cup\fb{\hj}$ is the covering just constructed, define the following sets of centers:
	\begin{align}
		\mD_{h}&\doteq \brcs*{x\in\ball{0}{1}\stset\ball{x}{\rho^h}\in\fa{h}},\quad 0\leq h\leq\hj-1\\
		\mD_{\hj}&\doteq\brcs*{x\in\ball{0}{1}\stset\ball{x}{\rho^{\hj}}\in\fb{\hj}}\\
		\mC&\doteq \mD_0\cup\dots\cup\mD_{\hj-1}\cup\mD_{\hj}\\		
		\mC_\ell&\doteq \mD_{\hj-\ell}\cup\dots\cup\mD_{\hj},\quad 0\leq \ell\leq \hj.
	\end{align}
Moreover, if $x\in\mC$, we'll also denote by $r_x$ the radius of the ball centered at $x$ which is contained in $\mF$. Notice that
	\begin{equation}
		\mC_\ell=\brcs*{x\in\mC\stset r_x\leq \rho^{\hj-\ell}},
	\end{equation}
and $\mC=\mC_{\hj}\subset\mC_{\hj-1}\subset\dots\subset\mC_{1}\subset\mC_{0}=\mD_{\hj}$.
\end{definition}

The next step is probably the most important of the whole construction: indeed, we show that \emph{we have a control on the $k^{\tx{th}}$ powers of the radii of the balls in $\mF$}. Here is where the refined techniques of \cref{section:reifenberg} become involved: we use Reifenberg Theorem \ref{theorem: Reifenberg} to achieve the final estimate, and \cref{theorem: L2 estimates} to check Reifenberg's hypothesis. Unfortunately, the proof is a bit intricate: we split it in several subtheorems.
\begin{remark}
From now on, we will assume \wlog{} that $\rho$ is of the form $5^{-\kappa}$ for some $\kappa\in\N$. This does not affect in any way the general procedure (at some point we will choose $\rho$ as an arbitrary number smaller than a certain constant) and simplifies a bit some computations.
\end{remark}

\begin{proposition}[Volume estimates]\label{proposition: volume estimates}
Let $\mF=\brcs*{\ball{x}{r_x}}_{x\in\mC}$ be the covering constructed in the previous paragraph. Recall that $\rho,\eta,\gamma, E>0$ are fixed constants. If $\gamma>0$ and $\rho>0$ are chosen small enough, there exists a constant $C_{\tx{I}}=C_{\tx{I}}(m,\rho)$ such that 
	\begin{equation}\label{equation: volume estimate}
		\sum_{x\in\mC}r_x^k\leq C_{\tx{I}}.
	\end{equation}
\end{proposition}

By Reifenberg Theorem \ref{theorem: Reifenberg}, the estimate $\eqref{equation: volume estimate}$ is achieved if the condition
	\begin{equation}\label{equation: volest Reif}
		\int_{\ball{w}{\tau}}\int_0^\tau \beta^k_\mu(y,s)^2\,\frac{ds}{s}\,d\mu(y)<\delta_R \tau^k
	\end{equation}
	holds for any ball $\ball{w}{\tau}$ with $w\in\ball{0}{1}$ and $0<\tau<1$ (or $0<\tau<\tau_{\tx{max}}$ for some $\tau_{\tx{max}}$, at the only price of worsening the constants involved). Here 
	$
		\mu\doteq\sum_{x\in\mC} r_x^k\delta_x.
	$
For any $0\leq h \leq\hj$, we now consider the measure $\mu_h$ associated to the set of centers $\mC_h$ defined in \cref{definition: centers}:
	$
		\mu_h\doteq \sum_{x\in\mC_h} r_x^k\delta_x.
	$
Clearly $\mu=\mu_{\hj}$; first of all, we state a very elementary \inquotes{induction property} of the measures $\mu_h$. 

\newcounter{volumeestimates}
\setcounter{volumeestimates}{1}
\renewcommand{\thesubtheorem}{\ref*{proposition: volume estimates}.\arabic{volumeestimates}}
\begin{subtheorem}\stepcounter{volumeestimates}\label{subtheorem: volest 1}
Let $h\in\brcs*{0,\dots,\hj-1}$. Let $C_{in}>0$ be a constant. Assume \wolog{} $\rho<\frac{1}{10}$. Assume that, for all $x\in\ball{0}{1}$ and all $s\in\bkts*{\quint \rho^{\hj}, \rho^{\hj-h}}$, it holds
	\begin{equation}\label{equation: estimate on muh}
		\mu_h\pths*{\ball{x}{s}}\leq C_{in}s^k.
	\end{equation}
Then there exists a constant $C_{f}$ depending only on $C_{in}$, $\rho$ and $m$ such that
	\begin{equation}
		\mu_{h+1}\pths*{\ball{x}{s}}\leq C_{f}s^k
	\end{equation}	
whenever one of the following holds:
	\begin{enumerate}[(i)]
	\item $\ball{x}{s}\cap \pths*{\mC_{h+1}\setminus\mC_h}=\emptyset$ and all $s\in\bkts*{\quint \rho^{\hj}, 2 \rho^{\hj-(h+1)}}$.
	\item $\ball{x}{s}$ contains a point of ${\mC_{h+1}\setminus\mC_h}$ and $s\in\bkts*{2 \rho^{\hj-h}, 2 \rho^{\hj-(h+1)}}$.
	\end{enumerate}
\end{subtheorem}

\begin{remark}
We could obviously state the same property with more general constants in front of the radii involved. This is however the form we will need: notice that the \inquotes{upper bound} for the radius gains a factor $2$. 
\end{remark}

\begin{proof}
Fix $x\in\ball{0}{1}$ and $s\in\bkts*{\quint \rho^{\hj}, 2 \rho^{\hj-(h+1)}}$. We can split $\mu_{h+1}\pths*{\ball{x}{s}}$ as 
	\begin{equation}
	\begin{split}
		\mu_{h+1}\pths*{\ball{x}{s}}&=\mu_h\pths*{\ball{x}{s}}+\sum_{\substack{z\in\mC_{h+1}\setminus\mC_h \\ z\in\ball{x}{s}}}\rho^{k\pths*{\hj-h-1}}\\
		&=\mu_h\pths*{\ball{x}{s}}+\rho^{k\pths*{\hj-h-1}}\card\pths*{\ball{x}{s}\cap \mC_{h+1}\setminus\mC_h}.
	\end{split}
	\end{equation}
Now:
\begin{itemize}
	\item If $\ball{x}{s}\cap\pths*{\mC_{h+1}\setminus\mC_h}=\emptyset$ and $s\in\bkts*{\quint \rho^{\hj}, \rho^{\hj-h}}$, then the first term is smaller or equal than $C_{in}s^k$ by assumption; the second term is trivially zero.
	\item If $s\in\bkts*{\rho^{\hj-h}, 2 \rho^{\hj-h-1}}$, then: we can cover $\ball{x}{s}\cap \supp(\mu_h)$ with a controlled number $c_1(\rho,m)$ of balls centered in $\mC_h$ with radius $\rho^{\hj-h}$, so that we obtain:
		\begin{equation}
			\mu_h\pths*{\ball{x}{s}}\leq c_1C_{in} \rho^{k(\hj-h)}\leq c_1C_{in} s^k;
		\end{equation}
	moreover, the number $\card\pths*{\ball{x}{s}\cap \pths*{\mC_{h+1}\setminus\mC_h}}$ is also bounded by a constant $c_2(\rho,m)$, because balls centered in $\pths*{\mC_{h+1}\setminus\mC_h}$ with radius $\quint \rho^{\hj-h-1}$ do not contain points of $\mC$ other then their center. Thus
		\begin{equation}
		\rho^{k\pths*{\hj-h-1}}\card\pths*{\ball{x}{s}\cap \mC_{h+1}\setminus\mC_h}\leq \dfrac{c_2}{\rho^k}s^k.
		\end{equation}
\end{itemize}
By choosing $C_{f}\doteq\max\brcs*{C_{in},c_1 C_{in}+\dfrac{c_2}{\rho^k}}$ we get the result.
\end{proof}

The next step is to prove that the estimate \cref{equation: estimate on muh} actually holds when $h=0$.

\begin{subtheorem}\stepcounter{volumeestimates}\label{subtheorem: volest 2}
There exists a constant $C_0(\rho,m)$ such that: for any $x\in\ball{0}{1}$ and $s\in\bkts*{\quint \rho^{\hj}, \rho^{\hj}}$,
\begin{equation}
	\mu_0\pths*{\ball{x}{s}}\leq C_0s^{-k}.
\end{equation}
\end{subtheorem}

\begin{proof}
An argument already used in \cref{subtheorem: volest 1}: if $x\neq y\in\mC_0$, then $\ball{x}{\quint\rho^{\hj}}$ and $\ball{y}{\quint\rho^{\hj}}$ are disjoint, thus the number of such centers contained in $\ball{x}{s}$ is bounded by a constant (the same $c_2(\rho,m)$ as in the previous proof). Thus
	\begin{equation}
		\mu_0\pths*{\ball{x}{s}}\leq c_2(\rho,m)\rho^{\hj k}\leq 5^kc_2 s^k,
	\end{equation}
which is what we needed.
\end{proof}

It may seem that, having an inductive step and a base step, we could already get the volume estimate we need. The problem is that we are applying \cref{subtheorem: volest 1} with an initial constant $C_{in}$ that keeps getting bigger at any step; instead, we would need in the end a universal constant that only depends on $\rho$ and $m$, since the number of steps is not fixed \textit{a priori}, and we don't want our constants to depend on it. Here is where Reifenberg Theorem comes into play.
The trick will be to prove that the estimate 
\begin{equation}
	\int_{\ball{w}{\tau}}\int_0^\tau \beta^k_{\mu_h}(y,s)^2\,\frac{ds}{s}\,d\mu_h(y)<\delta_R \tau^k
\end{equation}
holds for any $\mu_h$. 

\begin{subtheorem}\stepcounter{volumeestimates}\label{subtheorem: volest 3}
Let $\rho>0$ (small enough) and $C_{f}>0$ be fixed constants; $\eta, E, \mS$ as before. There exists a constant $\gamma=\gamma(\rho, C_{f}, m,p)$ such that the following holds. Assume that $\mF$ is the covering of $\mS$ associated to the constant $\gamma$, and that $\mu_{h}$ verifies the conclusion of \cref{subtheorem: volest 1}, \ie{}: the estimate
	\begin{equation}\label{eq:subth:volest3}
		\mu_{h}\pths*{\ball{x}{s}}\leq C_{f}s^k
	\end{equation}	
holds whenever one of the following holds:
	\begin{enumerate}[(i)]
	\item $\ball{x}{s}\cap \pths*{\mC_{h}\setminus\mC_{h-1}}=\emptyset$ and all $s\in\bkts*{\quint \rho^{\hj}, 2 \rho^{\hj-h}}$.
	\item $\ball{x}{s}$ contains a point of ${\mC_h\setminus\mC_{h-1}}$ and $s\in\bkts*{2 \rho^{\hj-\pths{h-1}}, 2 \rho^{\hj-h}}$.
	\end{enumerate} 
Then the following estimate is also true:
	\begin{equation}
	\mu_h\pths*{\ball{x}{s}}\leq C_Rs^{-k}
	\end{equation}
for all $x\in\ball{0}{1}$ and all $s\in\bkts*{\quint\rho^{\hj},\rho^{\hj-h}}$, where $C_R$ is the constant appearing in Reifenberg Theorem \ref{theorem: Reifenberg}.
\end{subtheorem}

\begin{proof}
We proceed in several steps. Let $h\in\brcs{0, \dots,\hj-1}$ be fixed.
\begin{stepizec}{Step}
\item \emph{(Application of Reifenberg Theorem)} Clearly, if we are able to prove that
\begin{equation}\label{equation: volest Reif 2}
	\int_{\ball{w}{\tau}}\int_0^\tau \beta^k_{\mu_h}(y,s)^2\,\frac{ds}{s}\,d\mu_h(y)<\delta_R \tau^k
\end{equation}
holds for all $w\in\ball{0}{1}$ and all $\tau<\rho^{\hj-h}$, then we can exploit the rescaled version of Reifenberg Theorem (\cref{corollary: rescaled reifenberg}), and we get exactly the thesis.
\item \emph{(Application of the Estimates on $\beta^k_\mu$)} \label{step: volest1}
	Notice that the integral with respect to $\mu_h$ appearing in \cref{equation: volest Reif 2} is actually a sum on $y\in\mC_h\cap \ball{w}{\tau}$. Let $y\in\mC_h$, and consider $\beta^k_{\mu_h}\pths*{y,s}$. Then:
	\begin{itemize}
	\item If $s\leq \quint r_y$, then $\beta^k_{\mu_h}\pths*{y,s}=0$, because $y$ is the only point of $\mC_h$ contained in $\ball{y}{s}$ (and thus any $k$-plane through $y$ is a best approximating plane for $\mu$);
	\item If $s\geq \quint r_y$, by property \ref{covering 1 property 4} of the covering $\mF$ (specifically \cref{equation: cov1prop4part2}), $u$ is not $\pths*{\half \eta,k+1}$-invariant in $\ball{y}{5\delta_0 s}$. Thus we can use \cref{theorem: L2 estimates} with $\bar{r}=5\delta_0$ and $\sigma=5$ (for example!) and obtain:
		\begin{equation}
		\beta^k_{\mu_h}(y,s)^2\leq C_J s^{-k}\int_{\ball{y}{s}}(\theta(z,5 s)-\theta(z, s))\,d\mu_h(z).
		\end{equation}
	where now $C_J$ depends on $m$, $p$ and $\eta$ only.
	\end{itemize}
More compactly, if we define the following function:
	\begin{equation}
	\begin{split}
	W(x,r)&\doteq \bkts*{\theta\pths*{x,5r}-\theta\pths*{x,r}}\chi_{\mC_h}(x)\chi_{\pths*{ r_x/5,\infty}}(r)=\\
	&=
		\begin{cases}
		\theta\pths*{x,5r}-\theta\pths*{x,r}	&\text{if $x\in\mC_h$ and $r\geq \quint r_x$}\\
		0										&\text{otherwise}
		\end{cases},
	\end{split}
	\end{equation}
then for all $s>0$ (smaller than a suitable constant) we have:
	\begin{equation}
		\beta^k_{\mu_h}(y,s)^2\leq C_J s^{-k}\int_{\ball{y}{s}}W(z,s)\,d\mu_h(z).
	\end{equation}
\item By \ref{step: volest1} and by Tonelli's Theorem, for a fixed $h\in\brcs*{0,\dots,\hj-1}$ we have:
	\begin{equation}
	\begin{split}
		\int_{\ball{w}{\tau}}\int_0^\tau \beta^k_{\mu_h}(y,s)^2\,\frac{ds}{s}\,d\mu_h(y)&\leq \int_{\ball{w}{\tau}}\int_0^\tau C_J s^{-k}\int_{\ball{y}{s}}W(z,s)\,d\mu_h(z)\,\frac{ds}{s}\,d\mu_h(y)\leq\\
		&\leq \int_0^\tau C_J s^{-k}\int_{\ball{w}{\tau}}\int_{\ball{y}{s}}W(z,s)\,d\mu_h(z)\,d\mu_h(y)\,\frac{ds}{s}.
	\end{split}
	\end{equation}	
	Notice that by the triangle inequality 
	\begin{equation}
		\abs{z-w}\leq\abs{z-y}+\abs{y-w}, 
	\end{equation}	
	so the set 
	\begin{equation}
	\brcs*{(y,z)\in\R^m\times\R^m\stset y\in\ball{w}{\tau},z\in\ball{y}{s}}
	\end{equation}
	is contained in 
	\begin{equation}
	\brcs*{(y,z)\in\R^m\times\R^m\stset z\in\ball{w}{\tau+s},y\in\ball{z}{s}}.
	\end{equation}
	Using again Tonelli Theorem, we also switch the two integrals in $\mu_{h+1}$, thus getting:
	\begin{equation}
	\begin{split}
		\int_{\ball{w}{\tau}}\int_0^\tau \beta^k_{\mu_h}(y,s)^2\,\frac{ds}{s}\,d\mu_h(y)&\leq \int_0^\tau  C_J s^{-k}\int_{\ball{w}{\tau+s}}\hspace{-10pt} W(z,s)
		\pths*{\int_{\ball{z}{s}}d\mu_h(y)}\,d\mu_h(z)\,\frac{ds}{s}\leq\\
		&\leq\int_0^\tau C_J s^{-k}\int_{\ball{w}{2\tau}}\hspace{-10pt} W(z,s)\,\mu_h\pths*{\ball{z}{s}}\,d\mu_h(z)\,\frac{ds}{s}.
	\end{split}
	\end{equation}
	Now for all the relevant pairs $(z,s)$ (for which $W$ is not $0$) the estimate \eqref{eq:subth:volest3} holds:
	\begin{equation}\label{eq:subth:volest3:eq2}
	\int_{\ball{w}{\tau}}\int_0^\tau \beta^k_{\mu_h}(y,s)^2\,\frac{ds}{s}\,d\mu_h(y)\leq
	C_J C_f\int_{\ball{w}{2\tau}}\int_0^\tau W(z,s)\,\frac{ds}{s}\,d\mu_h(z).
	\end{equation}
\item Let $x\in\mD_\ell$, so that $r_x=\rho^\ell=5^{-\kappa\ell}$ (see the Remark before the statement). The following estimate holds:\hiddennconst{C}{C:subth:volest3}
	\begin{equation}
	\begin{split}
		\int_{\quint r_x}^{\quint} \pths*{\theta\pths*{x,5 s}-\theta(x, s)}\,\frac{ds}{s}&=\sum_{j=1}^{\kappa\ell}\int_{\pths*{\quint}^{j+1}}^{\pths*{\quint}^j}\frac{{\theta\pths*{x,5s}-\theta(x, s)}}{s}\,ds\leq\\
		&\leq \sum_{j=1}^{\kappa\ell} \frac{\theta\pths*{x,5^{1-j}}-\theta\pths*{x,5^{-1-j}}}{5^{-1-j}}\pths*{\quint}^j\pths*{1-\pths*{\quint}}\leq\\
		&\leq C\bkts*{\theta(x,1)-\theta\pths*{x,5^{-\kappa\ell}}+\theta\pths*{x,{\quint}}-\theta\pths*{x,\quint 5^{-\kappa\ell}}}\leq\\
		&\leq \oconst{C}{C:subth:volest3}\gamma,
	\end{split}
	\end{equation}
	where the last inequality is a consequence of property \ref{covering 1 property 4} of the covering $\mF$ (specifically \cref{equation: cov1prop4part1}) and $\oconst{C}{C:subth:volest3}$ depends on $\rho$ and $\eta$. Plugging this information into \cref{eq:subth:volest3:eq2} (provided that $\tau\leq \quint$), we get
	\begin{equation}
	\int_{\ball{w}{\tau}}\int_0^\tau \beta^k_{\mu_h}(y,s)^2\,\frac{ds}{s}\,d\mu_h(y)\leq C_J C_f \oconst{C}{C:subth:volest3}\gamma\,\mu_h\pths*{\ball{w}{2\tau}}.
	\end{equation}
The left hand side is $0$ whenever $\ball{w}{2\tau}$ contains a single point of $\mC_h$; in all the other cases, the assumption \eqref{eq:subth:volest3} holds, thus
	\begin{equation}
	\int_{\ball{w}{\tau}}\int_0^\tau \beta^k_{\mu_h}(y,s)^2\,\frac{ds}{s}\,d\mu_h(y)\leq 2^{-k}C_J C_f^2 \oconst{C}{C:subth:volest3}\gamma \tau^{-k}.
	\end{equation}
Choosing $\gamma(\rho,C_f,m,p)\leq \frac{\delta_R(m)}{2^{-k}C_J C_f^2 \oconst{C}{C:subth:volest3}}$, we have the desired result.
\end{stepizec}
\end{proof}

We can finally prove \cref{proposition: volume estimates}.
\begin{proof}[Proof of \cref{proposition: volume estimates}] The proof is now a simple induction: by \cref{subtheorem: volest 2} we have an estimate on $\mu_0$ depending on a constant $C_0$; applying \cref{subtheorem: volest 1} the same estimate holds for $\mu_1$ with $C_f=C_f(C_0,m,p,\rho)$; but then we apply \cref{subtheorem: volest 3} to improve the constant: the estimate now holds for $\mu_1$ with $C_R(m)$. So we can repeat the procedure: the final constant for each $\mu_h$ will still be $C_R(m)$.
\end{proof}

\subsection{Second covering}
The goal now is to refine the covering in order to find balls which satisfy a clean energy drop; that is, we get rid in some sense of the sets of type $\mH(x,r)$, where the uniform energy drop does not happen, and which is already bound to lie in the fattening of a $(k-1)$-dimensional plane.
\paragraph{Construction of the second covering}
Consider again the \inquotes{first covering} $\mF$ for $\mS$. It is split in $\mF=\mG\cup\mE$, where $\mF,\mE$ and $\mG$ have the following properties:
	\begin{enumerate}
	\item Balls in $\mE\doteq\mE^{(0)}\doteq\mE_\hj$ have radius equal to $r=\rho^\hj$;
	\item Balls in $\mG\doteq\mG_0\cup\dots\cup\mG_{\hj-1}$ have radius $\rho^h$ with $h<\hj$; if $B=\ball{x}{r}\in\mG$, it satisfies the condition
		\begin{equation}
		\mH_B=\brcs*{y \in B\cap \mS\stset \theta\pths*{y,\quint\rho r}>E-\delta}\subset\fat{V_B}{\quint\rho r}
		\end{equation}
		for some $(k-1)$-affine subspace $V_B\in\aff{k-1}$.
	\item The estimate $\sum_{B\in\mF}r_B^k\leq C_{\tx{I}}(m)$ holds, where $r_B$ is the radius of $B$.
	\end{enumerate}
	
We now refine $\mF$ inductively, applying at each step a rescaled version of the procedure from \cref{subsection: first covering}.
\paragraph{Step 1.} Consider $B\in\mG$ and the associated $(k-1)$-plane $V_B$. We cover $B\cap\mS$ with balls of radius $\rho r_B$, divided in three subcoverings: $\mE_B$ (with radius \emph{equal} to $r$), $\mD_B$ (satisfying an \emph{energy drop} condition), $\mW_B$ (\emph{wild} balls on which we have no control).
	\begin{itemize}
	\item If $r_B=\rho^{\hj-1}$, simply cover $B\cap\mS$ with at most $\nconst{C}{C:covball}\pths*{m,\rho}$ balls of radius $\rho^\hj$. Call this covering $\mE_B$; set $\mD_B=\mW_B=\emptyset$. (Actually $\oconst{C}{C:covball}(m,\rho)=C(m)\rho^{-m}$, but it's irrelevant.)
	\item If $r_B>\rho^{\hj-1}$, we cover $\fat{\mH_B}{\quint \rho r_B}$ with at most $\nconst{C}{C:covfat}(m)\rho^{-(k-1)}$ balls of radius $\rho r_B$; call this covering $\mW_B$. This is possible since $\mH_B\subset V_B$; notice that the case $\mH_B=\emptyset$ is included.	Cover $\pths*{B\cap\mS}\setminus \fat{\mH_B}{\quint \rho r_B}$ with at most $\oconst{C}{C:covball}(m,\rho)$ balls of radius $\rho r_B$; call this covering $\mD_B$. Set $\mE_B=\emptyset$. Notice that if $\tilde{B}\in\mD_B$ then it satisfies the uniform energy drop condition.
	\end{itemize}
At this point we have a covering of $\mS$ of this type:
	\begin{equation}
		\mF^{\pths{1}}=\mE^{(1)}\cup\mD^{(1)}\cup\mW^{(1)},
	\end{equation}
where
	\begin{equation}
	\mE^{(1)}\doteq \mE^{(0)}\cup \bigcup_{B\in\mG}\mE_B, \qquad \mD^{(1)}\doteq\bigcup_{B\in\mG}\mD_B,\qquad \mW^{(1)}\doteq\bigcup_{B\in\mG}\mW_B,
	\end{equation}
and 
	\begin{alignat}{2}
	\sum_{B\in\mE^{(0)}}r_B^k&\leq C_{\tx{I}}(m),		
	&\sum_{B\in\mE^{(1)}\setminus\mE^{(0)}}r_B^k&\leq \rho^{k}\oconst{C}{C:covball}(m,\rho)C_{\tx{I}}(m)\\
	\sum_{B\in\mD^{(1)}}r_B^k&\leq \rho^{k}\oconst{C}{C:covball}(m,\rho)C_{\tx{I}}(m),\qquad	
	&\sum_{B\in\mW^{(1)}}r_B^k&\leq \oconst{C}{C:covfat}(m)C_{\tx{I}}(m)\rho^{k}\rho^{-k+1}.
	\end{alignat}
Introduce the constants
	\hiddennconst{K}{K:rhoandm}
	\hiddennconst{K}{K:onlym}
	\begin{align}
	\oconst{K}{K:rhoandm}(\rho,m) 	&=C_{\tx{I}}(m)\pths{1+2\rho^{k}\oconst{C}{C:covball}}\\
	\oconst{K}{K:onlym}(m)			&=C_{\tx{I}}(m)\oconst{C}{C:covfat}(m).\label{eq:Konlym}
	\end{align}
so that 
	\begin{equation}
	\sum_{\mE^{(1)}\cup\mD^{(1)}}r_B^k\leq \oconst{K}{K:rhoandm},\qquad \sum_{\mW^{(1)}}r_B^k\leq \oconst{K}{K:onlym}\rho.
	\end{equation}
\paragraph{Step $h+1$.}
Assume that, for some $h\leq \hj-1$, we have a covering of $\mS$ of the form $\mF^{\pths{h}}=\mE^{(h)}\cup\mD^{(h)}\cup\mW^{(h)}$ with the following properties:
	\begin{enumerate}
	\item\label{item:secondcov1} If $B\in\mE^{(h)}$, then $r_B=r=\rho^\hj$;
	\item\label{item:secondcov2} If $B\in\mD^{(h)}$, then the energy drop condition holds in $B$;
	\item\label{item:secondcov3} If $B\in\mW^{(h)}$, then $\rho^\hj<r_B\leq \rho^h$;
	\item The estimates
		\begin{equation}\label{eq:secondcovestim}
		\sum_{B\in\mE^{(h)}\cup\mD^{(h)}}r_B^k\leq\oconst{K}{K:rhoandm}\sum_{j=0}^{h-1}\pths*{\oconst{K}{K:onlym}\rho}^j,\qquad  \sum_{B\in\mW^{(h)}}r_B^k\leq\pths*{\oconst{K}{K:onlym}\rho}^h
		\end{equation}
		hold true.
	\end{enumerate}
\newcommand{\bst}{{B^\star}}
Consider a ball $\bst\in\mEh{h}$. Applying a rescaled version of the first construction (and of \cref{proposition: volume estimates}) we first find a covering $\mF_\bst$ for $\mS\cap \bst$ of the type
	\begin{equation}
	\mF_\bst=\fa{\bst;h}\cup\dots\cup\fa{\bst;\hj-1}\cup\fb{\bst;\hj}=\mG_\bst\cup\mE_\bst,
	\end{equation}
where $\mG_\bst$ are balls on which the energy drop condition is verified up to a neighborhood of a $(k-1)$-plane, $\mE_\bst$ are balls of radius $r=\rho^\hj$, and
	$\sum_{B\in\mF_\bst}r_B^k\leq C_{\tx{I}}(m)r_\bst^k$. Secondly, re-cover each ball of $\mG_\bst$ with a rescaled version of Step 1, thus obtaining
	\begin{equation}
	\mFh{h+1}_\bst=\mEh{h+1}_\bst\cup\mDh{h+1}_\bst\cup\mWh{h+1}_\bst,
	\end{equation}
where balls of $\mEh{h+1}_\bst$ have radius $r$, balls of $\mDh{h+1}_\bst$ satisfy the energy drop condition, balls of $\mWh{h+1}_\bst$ have radius $r<r_B\leq \rho^h+1$, and
	\begin{equation}
	\sum_{B\in\mEh{h+1}_\bst\cup\mDh{h+1}_\bst}r_B^k\leq \oconst{K}{K:rhoandm}r_\bst^k,\qquad \sum_{B\in\mWh{h+1}_\bst}r_B^k\leq \oconst{K}{K:onlym}\rho r_\bst^k.
	\end{equation}
Then define 
	\begin{gather}
	\mEh{h+1}\doteq\mEh{h}\cup\bigcup_{\bst\in\mWh{h}}\mEh{h+1}_\bst ,\qquad \mDh{h+1}\doteq\mDh{h}\cup\bigcup_{\bst\in\mWh{h}}\mDh{h+1}_\bst,\\
	\mWh{h+1}\doteq\bigcup_{\bst\in\mWh{h}}\mWh{h+1}_\bst,\\
	\mF^{\pths{h+1}}=\mE^{(h+1)}\cup\mD^{(h+1)}\cup\mW^{(h+1)}.
	\end{gather}
All the conditions \labelcref{item:secondcov1,item:secondcov2,item:secondcov3} are satisfied with $h+1$ instead of $h$; as for the estimates \eqref{eq:secondcovestim}, we have
	\begin{equation}
	\begin{split}
	\sum_{B\in\mEh{h+1}\cup\mDh{h+1}}\hspace{-8pt}r_B^k&=\sum_{B\in\mE^{(h)}\cup\mD^{(h)}}\hspace{-8pt}r_B^k+\sum_{\bst\in\mWh{h}}\sum_{B\in\mEh{h+1}_\bst\cup\mDh{h+1}_\bst}\hspace{-8pt}r_B^k\leq \\
	&\leq \oconst{K}{K:rhoandm}\sum_{j=0}^{h-1}\pths*{\oconst{K}{K:onlym}\rho}^j+ \oconst{K}{K:rhoandm}\hspace{-8pt}\sum_{\bst\in\mWh{h}}\hspace{-8pt}r_\bst^k\leq\oconst{K}{K:rhoandm} \sum_{j=0}^{h}\pths*{\oconst{K}{K:onlym}\rho}^j
	\end{split}	
	\end{equation}
and
	\begin{equation}
	\begin{split}
	\sum_{B\in\mWh{h+1}}r_B^k&=\sum_{\bst\in\mWh{h}}\sum_{B\in \mWh{h+1}_\bst}r_B^k\leq \\
	&\leq \oconst{K}{K:onlym}\rho\sum_{\bst\in\mWh{h}}r_\bst^k\leq \pths*{ \oconst{K}{K:onlym}\rho}^{h+1}.
	\end{split}
	\end{equation}
\par
Thus, as a consequence of this procedure, we have the following.

\begin{proposition}\label{prop: second covering}
Let $u\in\wonepN*{\Omega}$ be a $p$-minimizing map with energy bounded by $\Lambda$. Let $\eta>0$ be a constant and $1\leq k\leq m$. Assume that $E\leq \Lambda$ is \st{} $\theta(y,1)\leq E$ for all $y\in\ball{0}{1}\cap\mS$. Let $\mS\subset\strat{k}{\eta}{\delta_0 r}$ for some $r>0$. There exists a finite covering $\mF^\star$ of $\mS$ with the following properties:
	\begin{enumerate}[(i)]
		\item All the radii satisfy $r_x\geq r$;
		\item The $k^{\text{th}}$ powers of the radii are controlled by $\sum_{B\in\mF^\star}r_B^k\leq C_{\tx{II}}$, where $C_{\tx{II}}$ depends only on $m$.
		\item $\mF^\star=\mE^\star\cup\mD^\star$, where:
			\begin{enumerate}[(A)]
			\item For all $B\in\mE^\star$, $r_B=r$;
			\item Every ball $B\in\mD^\star$ satisfies a \emph{uniform energy drop} condition:
				\begin{equation}
				\theta\pths*{y,\quint r_B}<E-\delta\qquad\text{for all $y\in \mS\cap B$.}
				\end{equation}
			\end{enumerate}
	\end{enumerate}
\vspace{-8pt}
Here both $\delta_0$ and $\delta$ are constants that depend on $m,p,\varn,\Lambda,\eta$ (and nothing else).
\end{proposition}

\begin{proof}
Let $\bar{\rho}=\bar{\rho}(m)=\half \oconst{K}{K:onlym}(m)$, where $\oconst{K}{K:onlym}$ is the constant introduced in \eqref{eq:Konlym}. Once $\bar{\rho}(m)$ is chosen, also a constant $\bar{\gamma}(m,p)$ is fixed by \cref{subtheorem: volest 3}; as a consequence, $\delta_0(m,p)$ gets determined by \cref{lemma: tech lemma 1} and the constant ${\delta}(m,p)$ is fixed as well by \cref{lemma: tech lemma 2}. Assume \wolog{} that $r=\rho^\hj$ for some $\hj\in\N$. Perform the construction (of the first covering and then) of the second covering until Step $\hj$: then $\mWh{\hj}=\emptyset$ (by the bounds on the radii), so $\mFh{\hj}=\mEh{\hj}\cup\mDh{\hj}$; moreover, by \cref{eq:secondcovestim},
\hiddennconst{K}{K:final}
	\begin{equation}
		\sum_{B\in\mFh{\hj}}r_B^k\leq \oconst{K}{K:rhoandm}\pths*{m,\bar{\rho}(m)}\sum_{j=0}^{\hj-1}\pths*{\oconst{K}{K:onlym}\rho}^j\leq \oconst{K}{K:final}(m)\sum_{h=0}^\infty \pths*{\half}^h \leq 2\oconst{K}{K:final}(m).
	\end{equation}
This proves the proposition, by setting $C_{\tx{II}}(m)=2\oconst{K}{K:final}(m)$.
\end{proof}

\section{Proof of the main theorems}\label{section: proofs}
We split the main result \cref{theorem: main} in two parts, one concerning the estimate on the volume $\vol \pths*{\fat{\strat{k}{\eta}{\delta_0r}}{r}\cap\ball{0}{1}}$, and one for the rectifiability of $\mS^k_\eta$.
\subsection{Volume estimate}\label{subsection: volume estimate}
\begin{theorem}\label{theorem: final estimate}
Let $u\in\wonepN*{\Omega}$ be a $p$-minimizing map with energy bounded by $\Lambda$. Let $\eta>0$ and $1\leq k \leq m$.  There exists a constant $\oconst{C}{C:finalestimate}=\oconst{C}{C:finalestimate}(m,\varn,p,\Lambda,\eta)$ such that for any $r>0$
	\begin{equation}
	\vol \pths*{\fat{\strat{k}{\eta}{\delta_0r}}{r}\cap\ball{0}{1}}\leq \oconst{C}{C:finalestimate}r^{m-k}.
	\end{equation}
\end{theorem}

\hiddennconst{c}{c:ultimolemma}
The proof is a straightforward consequence of the following lemma:
\begin{lemma}\label{lemma: ultimo}
Let $m,p,\Lambda,\eta,k$ be constants, $u$ a map and $r>0$ as in \cref{theorem: final estimate}. For any number $i\in\N$ there exists a covering $\mF^\star_i$ of the set $\mS\doteq\strat{\eta}{\delta_0 r}{k}\cap\ball{0}{1}$ with the following properties:
\begin{enumerate}[(i)]
	\item The radii $r_B$ satisfy 
		\begin{equation}
			\sum_{B\in\mF^\star_i}r_B^k\leq (\oconst{c}{c:ultimolemma}(m)C_{\tx{II}}(m))^i
		\end{equation}
		for some new dimensional constant $\oconst{c}{c:ultimolemma}(m)$ and the old constant $C_{\tx{II}}(m)$ coming from \cref{prop: second covering};
	\item $\mF^\star_i=\mL^\star_i\cup\mD^\star_i$, where:
		\begin{enumerate}[(A)]
			\item $r_B\leq r$ for any $B\in\mL^\star_i$ (that is, $r_B$ is \emph{lower} or equal to the needed radius);
			\item\label{opt:ultB} For all $B\in\mD^\star_i$ and all $y\in\mS\cap B$, we have 
			\[\theta\pths*{y,\quint r_B}\leq \Lambda-i\delta.\]
		\end{enumerate}
\end{enumerate}
\end{lemma}

\begin{proof}
We proceed by induction on $i\in\N$. For $i=0$, we can  simply take $\mF^\star_0=\mD^\star_0=\brcs*{\ball{0}{1}}$. Assume then the lemma is true for some $i\geq 0$. Consider a ball $B_0\in\mD^\star_i$, and cover it with $\oconst{c}{c:ultimolemma}(m)$ balls of radius $\quint r_{B_0}$ (call $\mD_{i,B_0}$ this covering); for each of these balls $B$ consider the rescaling of $B$ (and $u$) through the transformation that maps it into the unit ball. Applying \cref{prop: second covering} with $E=\Lambda-i\delta$ and $\mS\setminus\bigcup_i\mL_i$, and scaling back to the original $B$, we find a covering $\mF^\star_B=\mE^\star_B\cup\mD^\star_B$ with
	\begin{enumerate}
	\item If $\tilde{B}\in\mE^\star_B$, then $r_{\tilde{B}}=r$;
	\item If $\tilde{B}\in\mD^\star_B$, then $\theta\pths*{y,\quint r_{\tilde{B}}}\leq\Lambda-i\delta-\delta$ for all $y\in\tilde{B}\cap \mS$;
	\item $\sum_{\tilde{B}\in\mF^\star_B}r_{\tilde{B}}^k\leq C_{\tx{II}}r_B^k$.
	\end{enumerate}	 
Thus, by defining
	\begin{align}
	\mD^\star_{i+1}&\doteq \bigcup_{B_0\in\mD_i^\star}\bigcup_{B\in\mD_{B_0,i}}\mD_B^\star\\
	\mE^\star_{i+1}&\doteq \mE^\star_{i}\cup \bigcup_{B_0\in\mD_i^\star}\bigcup_{B\in\mD_{B_0,i}}\mE_B^\star,
	\end{align}
we get the needed result.
\end{proof}
\newcommand{\strattwo}[2]{\mS_{#1}^{#2}}
\subsection{Rectifiability}
We now tackle the problem of the rectifiability of the strata of type $\strattwo{k}{\eta}$. It is clear that we'll need to use the second part of \cref{theorem: Reifenberg}; the technique is basically the same we used for the volume estimates, even with some simplifications.

\begin{theorem}\label{theorem: rectifiability}
Let $u\in\wonepN*{\Omega}$ be a $p$-energy minimizing map. For any $\eta>0$ and any $0\leq k\leq m$, the stratum $\mS^k_\eta(u)$ is $k$-rectifiable.
\end{theorem}

As we'll see shortly, the result follows easily from this \namecref{lemma: rect}.

\begin{lemma}\label{lemma: rect}
Let $m,p,\Lambda,\eta$ be fixed. There exist a universal constant $\kappa\pths*{m,p,\Lambda,\varn,\eta}$ with $0<\kappa<1$ such that the following holds. Let $u$ be $p$-minimizing, and let $\sing\subset\strattwo{\eta}{k}(u)\cap \ball{0}{1}$ be a $\haus^k$-measurable subset. There exists a $\haus^k$-measurable subset $\mathcal{R}\subset\sing$ with the following properties:
	\begin{enumerate}
		\item $\haus^k\pths*{\mathcal{R}}\leq \kappa \haus^{k}{\pths*{\sing}}$;
		\item The set $\sing\setminus \mathcal{R}$ is $k$-rectifiable.
	\end{enumerate}
\end{lemma}

Before proving this \lcnamecref{lemma: rect}, which requires some effort, we show how it is applied to prove \cref{theorem: rectifiability}.

\begin{proof}[Proof of \cref{theorem: rectifiability}]
By induction, for any $j\in\N$ there exists a $\haus^k$-measurable set $\mathcal{R}_j\subset\strattwo{\eta}{k}(u)$ \st{}:
\begin{itemize}
		\item $\haus^k\pths*{\mathcal{R}_j}\leq \kappa^j \haus^{k}\pths*{\strattwo{\eta}{k}(u)}$;
		\item The set $\sing\setminus \mathcal{R}_j$ is $k$-rectifiable.
\end{itemize}
This is easily proved: the step $j=1$ comes from the application of \cref{lemma: rect} to the stratum $\strattwo{\eta}{k}(u)$, while the $(j+1)^{\text{th}}$ step descends from the application of the same \namecref{lemma: rect} to $\mathcal{R}_j$. Now we can define
	\begin{align}
		\tilde{\mR}&\doteq \bigcap_{j\in\N}\mathcal{R}_j \\
		\tilde{\sing}&\doteq \strattwo{\eta}{k}(u)\setminus \tilde{\mathcal{R}}=\bigcup_{j\in\N}\pths*{\strattwo{\eta}{k}(u)\setminus\mathcal{R}_j}.
	\end{align}
Here $\tilde{\mR}$ has $\haus^k$-measure zero; and $\tilde{\mS}$ is the countable union of sets, each of which is countable union of Lipschitz $k$-graphs; therefore $\tilde{\mS}$ itself is a countable union of Lipschitz $k$-graphs. This means precisely that $\strattwo{\eta}{k}(u)$ is $k$-rectifiable.
\end{proof}

Now we turn to prove \cref{lemma: rect}.
\begin{proof}
We can assume that $\haus^{k}{\pths*{\sing}}>0$, otherwise the statement is trivial. 
\begin{stepizec}{Step}
\item	Consider the following map: for $x\in\ball{0}{1}$ and $r>0$ (small enough),
	\begin{equation}
		f_r(x)\doteq \theta\pths*{x,r}-\theta\pths*{x,0},
	\end{equation} 
	where $\theta\pths*{x,0}\doteq \lim_{s\to 0}\theta\pths*{x,s}$.
	As $r$ tends to $0$, the map $f_r$ converges pointwise and decreasingly to the constant function $f_0\equiv 0$; moreover, all the maps $f_r$ are bounded by the constant map $\Lambda$, which is integrable with respect to the measure $\haus^k\mres \sing$. Now fix a $\delta>0$. By the Dominated Convergence Theorem, there exists a $\bar{r}>0$ depending on $\delta$ \st{}
	\begin{equation}\label{eq:f4r}
		\int_{\sing}f_{5\bar{r}}(x)\,d\haus^k(x)\leq \delta^2\haus^{k}{\pths*{\sing}}.
	\end{equation}
	Consider the following sets:
	\begin{align}
		F_\delta&\doteq \brcs*{x\in\sing\stset f_{5\bar{r}(\delta)}(x)>\delta} \\
		G_\delta&\doteq \brcs*{x\in\sing\stset f_{5\bar{r}(\delta)}(x)\leq\delta}=\sing\setminus F_\delta;
	\end{align}
	observe that, since $f_{5\bar{r}}$ is nonnegative, we have:
	\begin{equation}
		\int_{\sing}f_{5\bar{r}}(x)\,d\haus^k(x)\geq \int_{F_\delta}f_{5\bar{r}}(x)\,d\haus^k(x)\geq \delta \haus^{k}{\pths*{F_\delta}};
	\end{equation}
	this, combined with \cref{eq:f4r}, gives
	\begin{equation}
		\haus^{k}{\pths*{F_\delta}}\leq \delta \haus^{k}{\pths*{\sing}}.
	\end{equation}
	We claim that, for $\delta$ sufficiently small, the set $G_\delta$ is $k$-rectifiable; if we manage to show this, then the \lcnamecref{lemma: rect} is proved. In order to prove this claim, we consider a \emph{finite} covering $\brcs*{\ball{x_i}{\bar{r}}}_{i=1}^L$ of $G_\delta$ made with balls of the fixed radius $\bar{r}(\delta)$. It is sufficient to show that for $\delta$ small $G_\delta\cap\ball{x_i}{\bar{r}(\delta)}$ is rectifiable for any $i$: our main aim will be now to check the applicability of the second Reifenberg Theorem (\cref{theorem: Reifenberg}, part \ref{item:reifenberg two}), that gives exactly that result.
\item	Fix a ball $\ball{x_i}{\bar{r}(\delta)}$, with $i\in\brcs*{1,\dots,L}$, and apply the usual transformation $\lambda_{x_i,\bar{r}}^{-1}$. We set
	\begin{equation}
		\tilde{u}=T_{x_i,\bar{r}}u,\qquad \tilde{G}_\delta=\lambda_{x_i,r}^{\leftarrow}\pths*{G_\delta}\cap\ball{0}{1}.
	\end{equation}
	Also, we define $\mu_\delta$ to be the measure $\haus^k\mres\tilde{G}_\delta$ on the unit ball $\ball{0}{1}$. Notice that for any $x\in\tilde{G}_\delta$ we have:
	\begin{equation}
		\theta^{\tilde{u}}(x,5)-\theta^{\tilde{u}}(x,0)\leq \delta,
	\end{equation}
	by the definition of $G_\delta$ and the usual scale invariance properties of $\theta$. Now the original $G_\delta$ was a subset of $\strattwo{\eta}{k}(u)$, hence $u$ was not $(\eta,k+1)$-invariant in $\ball{x}{\bar{r}s}$ for any point $x\in G_\delta$ and for any $s>0$; consequently, for any point $x$ in the transformed set $\tilde{G}_\delta$ and for any $s>0$, $\tilde{u}$ is not $(\eta,k+1)$-invariant in $\ball{x}{s}$. This is what we need to apply \cref{theorem: L2 estimates} on any ball $\ball{x}{s}$; and we apply it to the finite measure $\mu_\delta=\haus^k\mres \tilde{G}_\delta$. We obtain that, for any $x\in\tilde{G}_\delta$ and any $0<s\leq 1$,\hiddennconst{C}{C:L2est rect}
	\begin{equation}
		\beta^k_{\tilde{G}_\delta}(x,s)^2\leq \oconst{C}{C:L2est rect}\pths*{m,p,\eta} s^{-k}\int_{\ball{x}{s}}\theta\pths*{y,5s}-\theta\pths*{y,s}\,d\mu_\delta (y).
	\end{equation}
	This goes in the direction we need, since we are trying to check if Reifenberg condition \eqref{equation: reifenberg condition} is satisfied. Following what we did in the proof of \cref{proposition: volume estimates}, we first fix $w\in\ball{0}{1}$ and $r\leq 1$; for all $0<s\leq r$ we compute:
	\begin{equation}
		\int_{\ball{w}{r}}\beta^k_{\tilde{G}_\delta}(x,s)^2\,d\mu_\delta(x)\leq
		\oconst{C}{C:L2est rect} s^{-k} \int_{\ball{w}{r}}\pths*{\int_{\ball{x}{s}}\bkts*{\theta^{\tilde{u}}a\pths*{y,5s}-\theta^{\tilde{u}}\pths*{y,s}}\,d\mu_\delta(y)}\,d\mu_\delta(x)
	\end{equation}
	Observe that we are allowed to do this since $\mu_\delta$ is supported in $\tilde{G}_\delta$. As we have already noticed in \cref{proposition: volume estimates}, if $\abs{x-w}<r$ and $\abs{y-x}<s$, then $\abs{y-w}<r+s$: thus we can estimate
	\begin{equation}
	\begin{split}
		\int_{\ball{w}{r}}\beta^k_{\tilde{G}_\delta}(x,s)^2\,d\mu_\delta(x)&\leq
		\oconst{C}{C:L2est rect}s^{-k} \int_{\ball{w}{r+s}}\int_{\ball{y}{s}}\bkts*{\theta^{\tilde{u}}\pths*{y,5s}-\theta^{\tilde{u}}\pths*{y,s}}\,d\mu_\delta(x)\,d\mu_\delta(y)\leq\\
		&\leq \oconst{C}{C:L2est rect}s^{-k}\int_{\ball{w}{r+s}}\bkts*{\theta^{\tilde{u}}\pths*{y,5s}-\theta^{\tilde{u}}\pths*{y,s}}\haus^{k}{\pths*{\tilde{G}_\delta\cap\ball{y}{s}}}\,d\mu_\delta(y).
	\end{split}
	\end{equation}
	But now we can exploit the uniform volume estimates given by \cref{theorem: final estimate} (appropriately rescaled); we get the following uniform \textit{a priori} upper bound:
	\begin{equation}\label{eq:upperbound}
		\haus^{k}\pths*{\lambda_{x_i,\bar{r}}^{\leftarrow}\pths*{\strattwo{\eta}{k}(u)}\cap \ball{y}{s}}\leq \oconst{C}{C:finalestimate}\pths*{m,p,\varn,\Lambda,\eta}s^k;
	\end{equation}
	notice that thanks to this \textit{a priori} estimate it is not necessary to reproduce the induction argument of \cref{proposition: volume estimates}.
	Plugging this information in the previous inequality we get:
	\begin{equation}
		\int_{\ball{w}{r}}\beta^k_{\tilde{G}_\delta}(x,s)^2\,d\mu_\delta(x)\leq \oconst{C}{C:L2est rect}\oconst{C}{C:finalestimate}\int_{\ball{w}{r+s}}\bkts*{\theta^{\tilde{u}}\pths*{y,5s}-\theta^{\tilde{u}}\pths*{y,s}}\,d\mu_\delta(y).
	\end{equation}
	In order to check the validity of \cref{equation: reifenberg condition}, we now consider the \lhs{} of that inequality: applying Tonelli Theorem (twice), we find:
	\begin{equation}
	\begin{split}
		\int_{\ball{w}{r}}\pths*{\int_0^r \beta^k_{\tilde{G}_\delta}(x,s)^2\,\frac{ds}{s}}&\,d\mu_\delta(x)=\int_0^r\pths*{\int_{\ball{w}{r}} \beta^k_{\tilde{G}_\delta}(x,s)^2\,d\mu_\delta(x)}\,\frac{ds}{s}\leq\\
		&\leq \oconst{C}{C:L2est rect}\oconst{C}{C:finalestimate}\int_0^r\pths*{\int_{\ball{w}{2r}}\bkts*{\theta^{\tilde{u}}\pths*{y,5s}-\theta^{\tilde{u}}\pths*{y,s}}\,d\mu_\delta(y)}\,\frac{ds}{s}=\\
		&= \oconst{C}{C:L2est rect}\oconst{C}{C:finalestimate} \int_{\ball{w}{2r}}\pths*{\int_0^r \bkts*{\theta^{\tilde{u}}\pths*{y,5s}-\theta^{\tilde{u}}\pths*{y,s}}\,\frac{ds}{s}}\,d\mu_\delta(y).
	\end{split}
	\end{equation}
	Consider for a moment the inner integral; $r$ can simply be bounded by $1$. We use basically the same trick we exploited in \cref{proposition: volume estimates}:
	\hiddennconst{C}{C:only on five}
	\begin{equation}
	\begin{split}
		\int_0^1 \bkts*{\theta^{\tilde{u}}\pths*{y,5s}-\theta^{\tilde{u}}\pths*{y,s}}\,\frac{ds}{s}&=\sum_{j=0}^\infty\int_{5^{-(j+1)}}^{5^{-j}} \frac{\theta^{\tilde{u}}(y,5 s)-\theta^{\tilde{u}}(y, s)}{s}ds\leq\\
		&\leq\sum_{j=0}^\infty\int_{5^{-(j+1)}}^{5^{-j}} \frac{\theta^{\tilde{u}}(y,5^{-j+1})-\theta^{\tilde{u}}(y, 5^{-j-1})}{5^{-j-1}}ds\leq\\
		&\leq \oconst{C}{C:only on five} \sum_{j=0}^{\infty}\bkts*{\theta^{\tilde{u}}(y,5^{-j+1})-\theta^{\tilde{u}}(y, 5^{-j-1})}\leq \\
		&\leq  \oconst{C}{C:only on five}\bkts*{\pths*{\theta^{\tilde{u}}(y,5)-\theta^{\tilde{u}}(y,0)}+\pths*{\theta^{\tilde{u}}(y,1)-\theta^{\tilde{u}}(y,0)}}\leq\\
		&\leq 2 \oconst{C}{C:only on five}\delta.
	\end{split}
	\end{equation}
	\hiddennconst{C}{C:useless}
	Therefore we can insert this piece of information in the previous integral; using again the upper bound \eqref{eq:upperbound} on the measure of the singular stratum, we find, for a new constant $\nconst{C}{C:inverse of delta}(m,p,\varn,\Lambda,\eta)$:
	\begin{equation}
	\begin{split}
		\int_{\ball{w}{r}}\pths*{\int_0^r \beta^k_{\tilde{G}_\delta}(x,s)^2\,\frac{ds}{s}}\,d\mu_\delta(x)&\leq \oconst{C}{C:useless}\mu_\delta{\pths*{\ball{w}{2r}}}\delta\leq\\
		&\leq \oconst{C}{C:inverse of delta} \delta r^k.
	\end{split}
	\end{equation}
	Taking 
	\begin{equation}
	\delta<\frac{\delta_{\text{R}}(m)}{\oconst{C}{C:inverse of delta}(m,p,\varn,\Lambda,\eta)},
	\end{equation}
	we get exactly the hypothesis needed for the second part of Reifenberg Theorem: thus $\tilde{G}_\delta$ is $k$-rectifiable, and tracing back the steps of the proof this proves the $k$-rectifiability of $G_\delta$.
\end{stepizec}
\end{proof}
\printbibliography
\end{document}